\documentclass[10pt,a4paper]{amsart}
\usepackage[english]{babel}
\usepackage{url} 
\usepackage{latexsym}
\usepackage{amsmath}
\usepackage{amsfonts}
\usepackage{mathrsfs}
\usepackage{amstext}
\usepackage{amssymb}
\usepackage{amsthm}       
\usepackage{hyperref}
\usepackage{xspace}
\usepackage{graphicx,xcolor}

\usepackage[a4paper,top=2.1cm,bottom=2.60cm,left=3.8cm,right=3.8cm]{geometry} 

\newcommand{\R}{\mathbb{R}}
\newcommand{\N}{\mathbb{N}}

\newtheorem{teo}{Theorem}[section]
\newtheorem{defin}[teo]{Definition}

\newtheorem{prop}[teo]{Proposition}
\newtheorem{lemma}[teo]{Lemma}

\allowdisplaybreaks

\begin{document}

\title[Existence results for a nonlinear nonautonomus transmission problem]{Existence results for a nonlinear nonautonomus transmission problem via domain perturbation}

\date{}

\author{Matteo Dalla Riva}

\address[Matteo Dalla Riva]{Dipartimento di Ingegneria, Universit\`a degli Studi di Palermo, Viale delle Scienze, Ed. 8, 90128 Palermo, Italy.}
\email{matteo.dallariva@unipa.it}

\author{Riccardo Molinarolo}

\address[Riccardo Molinarolo]{Dipartimento di Matematica e Applicazioni ``Renato Caccioppoli'', Universit\`a degli Studi di Napoli Federico II, Via Cintia, Monte S. Angelo, 80216 Napoli, Italy.}
\email{riccardo.molinarolo@unina.it}

\author{Paolo Musolino}
\address[Paolo Musolino]{Dipartimento di Scienze Molecolari e Nanosistemi, Universit\`a Ca' Foscari Venezia, via Torino 155, 30170 Venezia Mestre, Italy.}
\email{paolo.musolino@unive.it}

\maketitle

\noindent 
{\bf Abstract:}   In this paper we study the existence and the analytic dependence upon domain perturbation of the solutions of a nonlinear nonautonomous transmission problem for the Laplace equation. The problem is defined in a pair of sets consisting of a perforated domain and an inclusion whose shape is determined by a suitable diffeomorphism $\phi$. First we analyse the case in which the inclusion is a fixed domain. Then we will perturb the inclusion and study the arising boundary value problem and the dependence of a specific family of solutions upon the perturbation parameter $\phi$.
\vspace{9pt}


\vspace{9pt}


\section{Introduction}
We begin by introducing the geometric framework of our problem. We fix once and for all a natural number 
\[
n \in \N \setminus \{0,1\}
\]
that will be the dimension of the Euclidean space $\R^n$ we are going to work in. We also fix a parameter
\[
\alpha \in ]0,1[\, ,
\] 
which we use to define the regularity of our sets and functions.
In order to introduce the domains where our problem is defined, we take two sets $\Omega^o$ and $\Omega^i$ that satisfy the following conditions:
\begin{equation}
	\begin{split}
		&\mbox{$\Omega^o$, $\Omega^i$ are bounded open connected subsets of $\R^n$ of class $C^{1,\alpha}$,} 
		\\
		&\mbox{with exteriors  $\R^n\setminus \overline{\Omega^o}$ and $\R^n\setminus \overline{\Omega^i}$ connected  and $\overline{\Omega^i}\subset \Omega^o$}
	\end{split}
\end{equation}
\begin{figure}[ht]
\centering
\includegraphics{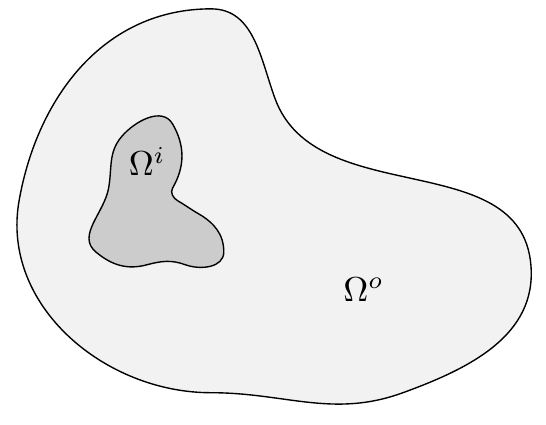}
\caption{\it The domains $\Omega^o$ and $\Omega^i$ ($n=2$)}\label{fig:Oe}
\label{introsetconditions}
\end{figure}
(see Figure \ref{introsetconditions}). Here the superscript  ``$o$'' stands for ``outer domain'' whereas the superscript ``$i$'' stands for ``inner domain.'' We first want to introduce a transmission problem in the pair of domains consisting of $\Omega^o \setminus \overline{\Omega^i}$ and $\Omega^i$.  Therefore, to define the boundary conditions, we fix three functions
\begin{equation}\label{introfunconditions}
F_1 \in C^0(\partial \Omega ^i \times \R \times \R),\qquad F_2 \in  C^0(\partial \Omega^i \times \R \times \R),\quad f^o \in C^{0,\alpha}(\partial\Omega^o).
\end{equation}
The functions $F_1$ and $F_2$ determine the transmission conditions on the inner boundary $\partial \Omega^i$. Instead, $f^o$ plays the role of the Neumann datum on the outer boundary $\partial \Omega^o$. We consider the following nonlinear transmission boundary value problem for a pair $(u^o,u^i) \in C^{1,\alpha}(\overline{\Omega^o} \setminus \Omega^i) \times C^{1,\alpha}(\overline{\Omega^i})$:
\begin{equation}\label{princeq}
	\begin{cases}
		\Delta u^o = 0 & \mbox{in } \Omega^o \setminus \overline{\Omega^i}, \\
		\Delta u^i = 0 & \mbox{in } \Omega^i, \\
		\nu_{\Omega^o}(x) \cdot \nabla u^o(x)=f^o(x) & \forall x \in \partial \Omega^o, \\
		\nu_{\Omega^i}(x) \cdot \nabla u^o (x) = F_1(x,u^o(x),u^i(x)) & \forall x \in \partial \Omega^i, \\
		\nu_{\Omega^i}(x) \cdot \nabla u^i (x) = F_2(x,u^o(x),u^i(x)) & \forall x \in \partial \Omega^i,
	\end{cases}
\end{equation}
where $\nu_{\Omega^o}$ and $\nu_{\Omega^i}$ denote the outward unit normal vector field to $\partial \Omega^o$ and to $\partial  \Omega^i$, respectively.
We note that, \emph{a priori}, it is not clear why problem $(\ref{princeq})$ should admit a classical solution. As a first result, we prove that under suitable conditions on $F_1$ and $F_2$, problem \eqref{princeq} has at least one solution $(u^o,u^i) \in C^{1,\alpha}(\overline{\Omega^o} \setminus \Omega^i) \times C^{1,\alpha}(\overline{\Omega^i})$. We notice that a problem similar to \eqref{princeq} has been studied in Dalla Riva and Mishuris \cite{DaMi15}. More precisely, in \cite{DaMi15} the authors consider a nonlinear transmission problem with a Dirichlet boundary condition on $\partial\Omega^o$ and a jump type condition for the normal derivative across the interface $\partial\Omega^i$. Then we study the existence and the analytic dependence of the solutions of the transmission problem \eqref{princeq} upon domain perturbation of the inclusion, {i.e.} of the inner set $\Omega^i$. Hence, we introduce a ``perturbed'' version of  problem \eqref{princeq}: we fix the external domain $\Omega^o$ and we assume that the boundary of the internal domain is of the form $\phi (\partial\Omega^i)$, where $\phi$ is a diffeomorphism of $\partial\Omega^i$  into a subset of $\R^n$ that belongs to the class
\begin{equation}\label{A_Omega^i}
\begin{split}
\mathcal{A}_{\partial\Omega^i} \equiv \Big\{ \phi & \in  C^{1,\alpha}(\partial\Omega^i, \R^n): \, \phi \text{ is injective and}\\
&\text{the differential} \, d\phi(y) \text{ is injective for all } y \in \partial\Omega^i  \Big\}.
\end{split}
\end{equation}
Clearly, the identity function of $\partial\Omega^i$ belongs to the class $\mathcal{A}_{\partial\Omega^i}$, and, for convenience, we set 
\begin{equation}\label{phi0}
\phi_0 \equiv \text{id}_{\partial\Omega^i}.
\end{equation}
Then by the Jordan Leray Separation Theorem (cf., {e.g.}, Deimling \cite[Theorem 5.2]{De85} and \cite[\S A.4]{DaLaMu21}), $\R^n \setminus \phi(\partial\Omega^i)$ has exactly two open connected
components for all $\phi \in  \mathcal{A}_{\partial\Omega^i}$, and we define $\Omega^i[\phi]$ to be the unique bounded open
connected component of $\R^n \setminus \phi(\partial\Omega^i)$. We set
\[
\mathcal{A}^{\Omega^o}_{\partial\Omega^i} \equiv \left\{ \phi \in \mathcal{A}_{\partial\Omega^i} : \overline{\Omega^i[\phi]} \subset \Omega^o  \right\}\, .
\]
By assumption \eqref{introsetconditions}, $\phi_0 \in \mathcal{A}^{\Omega^o}_{\partial\Omega^i}$. Now let $\phi \in \mathcal{A}^{\Omega^o}_{\partial\Omega^i}$. We wish to consider the following nonlinear transmission boundary value problem for a pair of functions $(u^o,u^i) \in C^{1,\alpha}(\overline{\Omega^o} \setminus \Omega^i[\phi]) \times C^{1,\alpha}(\overline{\Omega^i[\phi]})$:
\begin{equation}\label{princeqpertu}
\begin{cases}
\Delta u^o = 0 & \mbox{in } \Omega^o \setminus \overline{\Omega^i[\phi]}, \\
\Delta u^i = 0 & \mbox{in } \Omega^i[\phi], \\
\nu_{\Omega^o}(x) \cdot \nabla u^o(x)=f^o(x) & \forall x \in \partial \Omega^o, \\
\nu_{\Omega^i[\phi]}(x) \cdot \nabla u^o (x) = F_1(\phi^{(-1)}(x),u^o(x),u^i(x)) & \forall x \in \phi(\partial\Omega^i), \\
\nu_{\Omega^i[\phi]}(x) \cdot \nabla u^i (x) = F_2(\phi^{(-1)}(x),u^o(x),u^i(x)) & \forall x \in \phi(\partial\Omega^i),
\end{cases}
\end{equation}
where $\nu_{\Omega^i[\phi]}$ denotes the outward unit normal vector field to $\Omega^i[\phi]$. We prove that, under suitable conditions, problem \eqref{princeqpertu} admits a family of solutions $\{(u^o_\phi,u^i_\phi)\}_{\phi \in Q_0}$, where $Q_0$ is a neighbourhood of $\phi_0$ in $\mathcal{A}^{\Omega^o}_{\partial\Omega^i}$ and $(u^o_\phi,u^i_\phi)  \in C^{1,\alpha}(\overline{\Omega^o} \setminus \Omega^i[\phi]) \times C^{1,\alpha}(\overline{\Omega^i[\phi]})$ for every $\phi \in Q_0$. In literature, the existence of solutions of nonlinear boundary value problems has been largely investigated by means of variational techniques (see, {e.g.}, the monographs of Ne\v{c}as \cite{Ne83} and of Roub\'i\v{c}ek \cite{Ro13} and the references therein).
Moreover, potential theoretic techniques have been widely exploited to study nonlinear boundary value problems with  transmission conditions by Berger, Warnecke, and Wendland \cite{BeWaWe90}, by Costabel and Stephan \cite{CoSt90}, by Gatica and Hsiao \cite{GaHs95}, and by Barrenechea and Gatica \cite{BaGa96}. Boundary integral methods have been applied also by Mityushev and Rogosin for the analysis of transmission problems in the plane (cf.~\cite[Chap.~5]{MiRo00}).  
 Several authors have investigated the dependence upon domain perturbation of the solutions to boundary value problems and it is impossible to provide a complete list of contributions. Here we mention, for example,  Henrot and Pierre \cite{HePi05}, Henry \cite{He05}, Keldysh \cite{Ke66}, Novotny and Soko\l owski \cite{NoSo13}, and   Soko\l owski and Zol\'esio \cite{SoZo92}. Most of the contributions on this topic deal with first or second order shape derivability of functionals associated to the solutions of linear boundary value problems. 
In the present paper, instead,  we are interested into higher order regularity properties (namely real analiticity) of the solutions of a nonlinear problem. To do so, we choose to adopt the Functional Analytic Approach, which has
revealed to be a  powerful tool to analyse perturbed linear and nonlinear boundary value problems. This method has been first applied to investigate regular and singular domain perturbation problems for elliptic equations and systems with the aim of proving real analytic dependence upon the perturbation parameter (cf. Lanza de Cristoforis \cite{La02,La07-2,La10}). An application to the study of the behaviour of the effective conductivity of a periodic two-phase composite upon perturbations of the inclusion can be found in Luzzini and Musolino \cite{LuMu20}. The key point of the strategy of the method is the transformation of the perturbed boundary value problem into an equivalent functional equation that can be studied by the Implicit Function Theorem. Typically, such a transformation is achieved by exploiting classical results of potential theory, for example integral representation of harmonic functions in terms of layer potentials. Nonlinear transmission problems in perturbed domains have been studied by Lanza de Cristoforis in \cite{La10} and by the authors of the present paper in \cite{DaMoMu19, Mo19}, where they have investigated the behavior of the solution of a nonlinear transmission problem for the Laplace equation in a domain with a small inclusion shrinking to a point.

The paper is organised as follows. In Section \ref{notation} we define some of the symbols used later on. In Section \ref{Preliminaries} we introduce some classical results of potential theory that we need. Section \ref{sec princeq} is devoted to the study of problem \eqref{princeq}. 
We first prove a representation result for harmonic functions in $\overline{\Omega^o}\setminus\Omega$ and $\overline{\Omega}$ (where $\Omega$ is an open bounded connected subset of class $C^{1,\alpha}$ contained in $\Omega^o$) in terms of single layer potentials with appropriate densities and constant functions (cf. Lemma \ref{rapprharm}). Then we prove an uniqueness result in $C^{1,\alpha}(\Omega^o\setminus\overline{\Omega^i}) \times C^{1,\alpha}(\overline{\Omega^i})$ for an homogeneous linear transmission problem in the pair of domains $\Omega^o\setminus\overline{\Omega^i}$ and $\Omega^i$ and we analyse an auxiliary boundary operator arising from the integral formulation of that problem (cf. Lemma \ref{Alemma} and Proposition \ref{J_A}). In Proposition \ref{propintsys} we provide a formulation of problem \eqref{princeq} in terms of integral equations. 
The obtained integral system is solved by means of a fixed-point theorem, namely the Leray-Schauder Theorem (cf. Proposition \ref{Tcontcomp} and Proposition \ref{prop mu_0}). Finally, under suitable conditions on the functions $F_1$ and $F_2$, we obtain an existence results in $C^{1,\alpha}(\Omega^o\setminus\overline{\Omega^i}) \times C^{1,\alpha}(\overline{\Omega^i})$ for problem \eqref{princeq} (cf. Proposition \ref{prop u^o_0,u^i_0}).
Section \ref{sec princeqpertu} is devoted to the study of problem \eqref{princeqpertu}. We provide a formulation of problem \eqref{princeqpertu} in terms of integral equations depending on the diffeomorphism $\phi$ which we rewrite into an equation of the type $M[\phi,\mu] = 0$ for an auxiliary map $M: \mathcal{A}^{\Omega^o}_{\partial\Omega^i} \times X \to Y$ (with $X$ and $Y$ suitable Banach spaces), where the variable $\mu$ is related to the densities of the integral representation of the solution (cf. Proposition \ref{M=0prop}). Then, by analiticity results for the dependence of single and double layer potentials upon the perturbation of the support, we prove that $M$ is real analytic (cf. Proposition \ref{Mrealanal}) and the differential of $M$ with respect to the variable $\mu \in X$ is an isomorphism (cf. Proposition \ref{diffMprop}). Hence, by the Implicit Function Theorem, we show the existence of a family of solutions $\{(u^o_\phi,u^i_\phi)\}_{\phi \in Q_0}$ of  \eqref{princeqpertu} (cf. Theorem \ref{upertuex}) and we prove that it can be represented in terms of real analytic functions (cf. Theorem \ref{upertuana}).

\section{Notation}\label{notation}
We denote by $\N$ the set of natural numbers including $0$. We denote the norm of a real normed space $X$ by $\| \cdot \| _X$. We denote by $I_X$ the identity operator from $X$ to itself and we omit the subscript $X$ where no ambiguity can occur. If $X$ and $Y$ are normed spaces we  consider on the product space $X \times Y$ the norm defined by $\| (x,y) \|_{X \times Y} \equiv \|x\|_X + \|y\|_Y $ for all $(x,y) \in X \times Y$, while we use the Euclidean norm for $\R^d$, $d\in\mathbb{N}\setminus\{0,1\}$. If $U$ is an open subset of $X$, and $F:U \to Y$ is a Fr\'echet-differentiable map in $U$, we denote the differential of $F$ by $dF$.
The inverse function of an invertible function $f$ is denoted by $f^{(-1)}$, while the reciprocal of a non-zero scalar function $g$ or the inverse of an invertible matrix $A$ are denoted by $g^{-1}$ and $A^{-1}$ respectively. Let $\Omega \subseteq \R^n$. Then $\overline{\Omega}$ denotes the closure of $\Omega$ in $\R^n$, $\partial \Omega$ denotes the boundary of $\Omega$, and $\nu_\Omega$ denotes the outward unit normal to $\partial \Omega$.
For $x \in \R^d$, $x_j$ denotes the $j$-th coordinate of $x$, $|x|$ denotes the Euclidean modulus of $x$ in $\R^d$. If $x \in \mathbb{R}^d$ and $r>0$, we denote by $B_d(x,r)$ the open ball of center $x$ and radius $r$.  Let $\Omega$ be an open subset of $\R^n$ and $m \in \N \setminus \{0\}$. The space of $m$ times continuously differentiable real-valued function on $\Omega$ is denoted by $C^m(\Omega,\R)$ or more simply by $C^m(\Omega)$. Let $r \in \N \setminus \{0\}$, $f \in (C^m(\Omega))^r$. The $s$-th component of $f$ is denoted by $f_s$ and the gradient of $f_s$ is denoted by $\nabla f_s$. Let $\eta=(\eta_1, \dots ,\eta_n) \in \N^n$ and $|\eta|=\eta_1+ \dots+\eta_n$. Then $D^\eta f \equiv \frac{\partial^{|\eta|}f}{\partial x^{\eta_1}_1, \dots , \partial x^{\eta_n}_n}$. 
We retain the standard notation for the space $C^{\infty}(\Omega)$ and its subspace $C^{\infty}_c(\Omega)$ of functions with compact support. The subspace of $C^m(\Omega)$ of those functions $f$ such that $f$ and its derivatives $D^\eta f$ of order $|\eta|\le m$ can be extended with continuity to $\overline{\Omega}$ is denoted $C^m(\overline{\Omega})$. 
We denote by $C^m_b(\overline{\Omega})$ the space of functions of $C^m(\overline{\Omega})$ such that $D^{\eta} f$ is bounded for $|\eta|\leq m$. Then the space $C^m_b(\overline{\Omega})$ equipped with the usual norm $\|f\|_{C^m_b(\overline{\Omega})} \equiv \sum_{|\eta|\leq m} \sup_{\overline{\Omega}} |D^\eta f|$ is well known to be a Banach space. Let $f \in C^0(\overline{\Omega})$. Then we define its H\"{o}lder constant as
\begin{equation*}
|f : \Omega|_\alpha\equiv \mbox{sup} \left\{\frac{|f(x)-f(y)|}{|x-y|^\alpha} : x,y \in \overline{\Omega}, x \neq y \right\}.
\end{equation*}
We define the subspace of $C^0(\overline{\Omega})$ of H\"{o}lder continuous functions with exponent $\alpha \in ]0,1[$ by
$C^{0,\alpha}(\overline{\Omega}) \equiv \{f \in C^0(\overline{\Omega}) : \, |f : \Omega|_\alpha < \infty \}$. Similarly, the subspace of $C^m(\overline{\Omega})$ whose functions have $m$-th order derivatives that are H\"{o}lder continuous with exponent $\alpha \in ]0,1[$ is denoted $C^{m,\alpha}(\overline{\Omega})$. Then the space 
$C^{m,\alpha}_b(\overline{\Omega}) \equiv C^{m,\alpha}(\overline{\Omega}) \cap  C^m_b(\overline{\Omega}) \,,$
equipped with its usual norm $\|f\|_{C^{m,\alpha}_b(\overline{\Omega})} \equiv \|f\|_{C^{m}_b(\overline{\Omega})} + \sum_{|\eta|=m}{|D^\eta f : \Omega|_\alpha}$, is a Banach space. If $\Omega$ is bounded, then $C^{m,\alpha}_b(\overline{\Omega}) = C^{m,\alpha}(\overline{\Omega})$, and we omit the subscript $b$.
We denote by $C^{m,\alpha}_{\mathrm{loc}}(\R^n \setminus \Omega)$ the space of functions on $\R^n \setminus \Omega$ whose restriction to $\overline{U}$ belongs to $C^{m,\alpha}(\overline{U})$ for all open bounded subsets $U$ of $\R^n \setminus \Omega$. On $C^{m,\alpha}_{\mathrm{loc}}(\R^n \setminus \Omega)$ we consider the natural structure of Fr\'echet space. Finally if $\Omega$ is bounded, we set
\begin{equation*}
C^{m,\alpha}_{\mathrm{h}}(\overline{\Omega}) \equiv \{ u \in C^{m,\alpha}(\overline{\Omega}) \cap C^2(\Omega): \Delta u = 0 \text{ in } \Omega \},
\end{equation*} 
\begin{equation*}
\begin{split}
C^{m,\alpha}_{\mathrm{h}}(\mathbb{R}^n \setminus \Omega) \equiv \{ u \in C^{m,\alpha}(\mathbb{R}^n \setminus \Omega)&  \cap C^2(\mathbb{R}^n \setminus \overline{\Omega}):  \Delta u = 0 \text{ in } \mathbb{R}^n \setminus \overline{\Omega},   \\
& |u(x)| = O(|x|^{2-n}) \text{ as } x \to +\infty \}.
\end{split}
\end{equation*}
The condition $|u(x)| = O(|x|^{2-n})$  as $x \to +\infty$ in the above definition is equivalent for an harmonic function to the so-called harmonicity at infinity (see Folland \cite[Prop.~(2.74), p.~112]{Fo95}). We say that a bounded open subset of $\R^n$ is of class $C^{m,\alpha}$ if it is a manifold with boundary imbedded in $\R^n$ of class $C^{m,\alpha}$. In particular if $\Omega$ is a $C^{1,\alpha}$ subset of $\R^n$, then $\partial\Omega$ is a $C^{1,\alpha}$ sub-manifold of $\R^n$ of co-dimension $1$.  
If $M$ is a $C^{m,\alpha}$ sub-manifold of $\R^n$ of dimension $d\ge 1$, we define the space $C^{m,\alpha}(M)$ by exploiting a finite local parametrization. 
 We retain the standard definition of the Lebesgue spaces $L^p$, $p\ge 1$. If $\Omega$ is of class $C^{1,\alpha}$, we denote by $d\sigma$ the area element on $\partial\Omega$. If $Z$ is a subspace of $L^1(\partial \Omega)$, we set  
\[
Z_0 \equiv \left\{ f \in Z : \int_{\partial\Omega} f \,d\sigma = 0 \right\}.
\]
Then we introduce a notation for superposition operators: if $H$ is a function from $\partial \Omega^i \times\R\times \R$ to $\R$, then we denote by $\mathcal{N}_{H}$ the nonlinear nonautonomous superposition operator that take a pair $(h^1,h^2)$ of functions from $\partial\Omega^i$ to $\R$ to the function $\mathcal{N}_{H}(h^1,h^2)$ defined by
\begin{equation*}
\mathcal{N}_{H}(h^1,h^2)(x) \equiv H(x,h^1(x),h^2(x)) \quad\forall x \in \partial\Omega^i.
\end{equation*}
Here the letter ``$\mathcal{N}$'' stands for ``Nemytskii operator.'' Finally, we have the following   by Lanza de Cristoforis and Rossi \cite[Lemma 3.3, Prop. 3.13]{LaRo04}.

\begin{lemma}\label{lemmanotation}
	Let $\Omega^i$ be as in \eqref{introsetconditions} and let $\mathcal{A}_{\partial\Omega^i}$ be as in \eqref{A_Omega^i}. Let $\phi \in \mathcal{A}_{\partial\Omega^i}$. Then there exists a unique function $\tilde{\sigma}_n[\phi] \in C^{0,\alpha}(\partial\Omega^i)$ such that
	\[
	\int_{\phi(\partial\Omega^i)} f(y) \,d\sigma_y = \int_{\partial\Omega^i} f(\phi(s)) \, \tilde{\sigma}_n[\phi](s) \,d\sigma_s \quad\forall f \in L^1(\phi(\partial\Omega^i)).
	\]
	Moreover the map from $\mathcal{A}_{\partial\Omega^i}$ to $C^{0,\alpha}(\partial\Omega^i)$  that  takes $\phi$ to $\tilde{\sigma}_n[\phi]$ and the map from  $\mathcal{A}_{\partial\Omega^i}$ to $C^{0,\alpha}(\partial\Omega^i)$  that  takes $\phi$ to $\nu_{\Omega^i[\phi]}(\phi(\cdot))$ are real analytic.
\end{lemma}

\section{Some preliminaries of potential theory}\label{Preliminaries}

As we have mentioned, a key point of the  Functional Analytic Approach is the reformulation of the boundary value problem in terms of an equivalent integral equation. To this aim, we exploit representation formulas for harmonic functions in terms of layer potentials. In this section, we collect some classical results of potential theory. We do not present proofs that can be found, for example, in Folland \cite[Chap. 3]{Fo95}, in Gilbarg and
Trudinger \cite[Sec. 2]{GiTr83}. 
\begin{defin}
	We denote by $S_n$ the function from $\R^n \setminus \{0\}$ to $\R$ defined by
	\begin{equation*}
	S_n(x) \equiv 
	\begin{cases}
	\frac{1}{s_n} \log |x| & \forall x \in \R^n \setminus \{0\} \qquad \mbox{if } n=2 \\
	\frac{1}{(2-n) s_n} |x|^{2-n} & \forall x \in \R^n \setminus \{0\} \qquad \mbox{if } n>2
	\end{cases}
	\end{equation*}
	where $s_n$ denotes the $(n-1)$-dimensional measure of $\partial B_n(0,1)$.
\end{defin}    

$S_n$ is well known to be a fundamental solution of the Laplace operator $\Delta=\sum_{j=1}^n\partial^2_{x_j}$.
We now assume that
\[
 \text{$\Omega$ is an open bounded subset of $\R^n$ of class $C^{1,\alpha}$}. 
\]
 In the following definition we introduce the single layer potential, which we use to transform our problems into integral equations.

\begin{defin}
	We denote by $v_{\Omega}[\mu]$ the single layer potential with density $\mu$, {i.e.} the function defined by
	\begin{equation*}
	v_{\Omega}[\mu](x) \equiv \int_{\partial \Omega}{S_n(x-y) \mu(y) \,d\sigma_y} \qquad \forall x \in \R^n \, , \forall \mu \in L^2(\partial\Omega)\, .
	\end{equation*}
\end{defin}
It is well known that if $\mu \in C^{0,\alpha}(\partial\Omega)$, then $v_{\Omega}[\mu] \in C^0(\R^n)$. We set 
\begin{equation*}
v^+_{\Omega}[\mu]\equiv v_{\Omega}[\mu]_{| \overline{\Omega}}, \qquad v^-_{\Omega}[\mu] \equiv v_{\Omega}[\mu]_{| \R^n \setminus \Omega}.
\end{equation*}

Then we define the boundary integral operators associated to the trace of  the single layer potential and its normal derivative.

\begin{defin}\label{defV-W}
	We denote by $V_{\partial\Omega}$ the operator from $L^2(\partial\Omega)$ to itself  that  takes $\mu$ to the function $V_{\partial\Omega}[\mu]$ defined in the trace sense by
	\[
	V_{\partial\Omega}[\mu]\equiv v_{\Omega}[\mu]_{|\partial\Omega}.
	\]
	We denote by $W_{\partial\Omega}$ the integral operator from $L^2(\partial \Omega)$ to itself defined by 
	\begin{equation*}
	W_{\partial\Omega}[\mu](x) \equiv - \int_{\partial \Omega}\! \!{\nu_\Omega(y) \cdot \nabla S_n(x-y) \mu(y) \,d\sigma_y} \ \  \text{for a.e. } x \in \partial \Omega, \forall \mu \in L^2(\partial\Omega).
	\end{equation*}
	We denote by $W^\ast_{\partial\Omega}$ the  integral operator from $L^2(\partial \Omega)$ to itself which is the transpose of $W_{\partial\Omega}$ and  that  is defined by
	\[
	W^\ast_{\partial\Omega}[\mu](x) \equiv \int_{\partial \Omega}\! \!{\nu_\Omega(x) \cdot \nabla S_n(x-y) \mu(y) \,d\sigma_y} \ \  \mbox{for a.e. } x \in \partial \Omega, \forall \mu \in L^2(\partial\Omega).
	\]
\end{defin}

As it is well known, since $\Omega$ is of class $C^{1,\alpha}$, $W_{\partial\Omega}$ and $W^\ast_{\partial\Omega}$ are compact operators from $L^2(\partial\Omega)$ to itself (both display a weak singularity). In particular $\left( \pm\frac{1}{2} I + W^\ast_{\partial\Omega} \right)$ are Fredholm operators of index $0$ from $L^2(\partial\Omega)$ to itself. Moreover, one verifies that $W_{\partial\Omega}: C^{1,\alpha}(\partial\Omega) \to C^{1,\alpha}(\partial\Omega)$ and $W^\ast_{\partial\Omega}: C^{0,\alpha}(\partial\Omega) \to C^{0,\alpha}(\partial\Omega)$ are transpose to one another with respect to the duality of $C^{1,\alpha}(\partial\Omega) \times C^{0,\alpha}(\partial\Omega)$ induced by the inner product of $L^2(\partial\Omega)$. We collect some well known properties of the single layer potential in the theorem below. In particular, we note that the operator of statement (iv) is an isomorphism both in the case of dimension $n=2$ and $n \geq 3$ (see, {e.g.}, \cite[Theorem 6.47]{DaLaMu21}).

\begin{teo}[Properties of the single layer potential]\label{sdp}
	The following statements hold.
	\begin{enumerate}
		\item[(i)] For all $\mu \in L^2(\partial\Omega)$, the function $v_{\Omega}[\mu]$ is harmonic in $\R^n\setminus \partial\Omega$. If $n\geq3$ or if $n=2$ and $\int_{\partial \Omega}\mu\, d\sigma=0$ then $v_{\Omega}[\mu]$ is also harmonic at infinity.
		
		\item[(ii)] If $\mu \in C^{0,\alpha}(\partial\Omega)$, then $v^+_{\Omega}[\mu] \in C^{1,\alpha}(\overline{\Omega})$ and the map from $C^{0,\alpha}(\partial\Omega)$ to $C^{1,\alpha}(\overline{\Omega})$  that  takes $\mu$ to $v^+_{\Omega}[\mu]$ is linear and continuous. Moreover, $v^-_{\Omega}[\mu] \in C^{1,\alpha}_{\mathrm{loc}}(\mathbb{R}^n \setminus \Omega)$ and the map from $C^{0,\alpha}(\partial \Omega)$ to $C^{1,\alpha}_{\mathrm{loc}}(\mathbb{R}^n \setminus \Omega)$  that  takes $\mu$ to $v^-_{\Omega}[\mu]$ is linear and continuous.
		
		\item[(iii)] If $\mu \in C^{0,\alpha}(\partial\Omega)$, then we have following jump relations
		\begin{equation*}
		\nu_\Omega(x) \cdot \nabla v^\pm_{\Omega}[\mu] (x) = \left( \mp \frac{1}{2} I + W^\ast_{\partial\Omega} \right)[\mu](x) \qquad \forall x \in \partial \Omega.
		\end{equation*}	
		
		\item[(iv)] The map from $C^{0,\alpha}(\partial\Omega)_0 \times \R$ to $C^{0,\alpha}(\partial\Omega)$  that  takes a pair $(\mu,\rho)$ to $V_{\partial\Omega}[\mu] + \rho$ is an isomorphism.
	\end{enumerate}
\end{teo}

Since $\Omega$ is of class $C^{1,\alpha}$, the following classical compactness result holds (cf.~Schauder \cite{Sc31,Sc32}).

\begin{teo}\label{Schaudercompact}
	The map  that  takes $\mu$ to $W^\ast_{\partial\Omega}[\mu]$ is compact from $C^{0,\alpha}(\partial\Omega)$ to itself.
\end{teo}

Theorem \ref{Schaudercompact} implies that $\left( \pm\frac{1}{2} I + W^\ast_{\partial\Omega} \right)$ are Fredholm operators of index $0$ from $C^{0,\alpha}(\partial\Omega)$ into itself.  We now collect some regularity results for integral operators. We first introduce the following (see Folland \cite[Chap. 3 \S B]{Fo95}).
\begin{defin}\label{kerneldef}
	Let $K$ be a measurable function from $\partial\Omega \times \partial\Omega$ to $\R$ and let $0 \leq \beta < n-1$. We say that $K$ is a continuous kernel of order $\beta$ if
	\begin{equation*}
	K(x,y) = k(x,y) |x-y|^{-\beta} \quad \forall(x,y)\in \partial\Omega \times \partial\Omega,
	\end{equation*}
	for some continuous function $k$ on $\partial\Omega \times \partial\Omega$.
	\\
	If $K$ is a continuous kernel of order $\beta$, we denote by  $\mathcal{K}_K$ the integral operator from $L^2(\partial\Omega)$ to itself defined by
	\begin{equation*}
	\mathcal{K}_K [\mu] (x) \equiv \int_{\partial\Omega} K(x,y) \mu (y) \,d\sigma_y \qquad \text{for a.e. } x \in \partial \Omega\, , \forall \mu \in L^2(\partial\Omega)\, .
	\end{equation*}
\end{defin}

We observe that the functions $K_1(x,y) \equiv S_n(x-y)$ and $K_2(x,y) \equiv \nu_\Omega(y) \cdot \nabla S_n(x-y)$ of $(x,y) \in \partial\Omega\times \partial \Omega$, $x \neq y$, are continuous kernels of order $n-2$ (cf. Folland \cite[Prop. 3.17]{Fo95}). Clearly, we can extend the notion of integral operator with a continuous kernel to the vectorial case  just applying the definition above component-wise. Then we present a vectorial version of a classical regularity result (see, for example, Folland \cite[Prop. 3.13]{Fo95}).

\begin{teo}\label{regularityvecttheorem}
	Let $0\leq \beta < n-1$. Let $K_i^j$ with $i,j \in \{1,2\}$ be continuous kernels of order $\beta$. Let $\mathcal{K}=(\mathcal{K}_1,\mathcal{K}_2)$ be the operator from $(L^2(\partial\Omega))^2$ to itself  defined by
\[
	\mathcal{K}_1 [\mu_1,\mu_2] = \mathcal{K}_{K_1^1}[\mu_1] +  \mathcal{K}_{K_2^1}[\mu_2],\qquad
	\mathcal{K}_2 [\mu_1,\mu_2] = \mathcal{K}_{K_1^2}[\mu_1] +  \mathcal{K}_{K_2^2}[\mu_2], 
\] 
	for all $(\mu_1,\mu_2) \in (L^2(\partial\Omega))^2$. If $(I + \mathcal{K})[\mu_1,\mu_2]  \in (C^{0}(\partial\Omega))^2$,	then $(\mu_1,\mu_2)  \in (C^{0}(\partial\Omega))^2$.
\end{teo}

Finally, we present in Theorem \ref{regularitytheorem}  a regularity result that will be widely used in what follows. The proof exploits a standard argument on iterated kernels and can be found, {e.g.},  in Dalla Riva and Mishuris \cite[Lem.~3.3]{DaMi15}.

\begin{teo}\label{regularitytheorem}
	Let $\mu \in L^2(\partial\Omega)$. Let $\beta \in [0,\alpha]$. If $\left( \frac{1}{2} I + W^\ast_{\partial\Omega} \right)[\mu]$ or $\left( -\frac{1}{2} I + W^\ast_{\partial\Omega} \right)[\mu]$ belongs to $C^{0,\beta}(\partial\Omega)$, then $\mu \in C^{0,\beta}(\partial\Omega)$.
\end{teo}

\section{Existence result for problem (\ref{princeq})}\label{sec princeq}

The aim of this section is to prove an existence result for problem \eqref{princeq}. We start with the following representation result for harmonic functions in $\Omega^o \setminus \overline{\Omega}$ and in $\Omega$ in terms of single layer potentials plus constant functions. The set $\Omega$ in the Lemma \ref{rapprharm} will be later replaced by the set $\Omega^i$ and by the perturbed set $\Omega^i[\phi]$.

\begin{lemma}\label{rapprharm}
	Let $\Omega$ be an open bounded connected subset of $\R^n$ of class $C^{1,\alpha}$, such that $\mathbb{R}^n\setminus \overline{\Omega}$ is connected and $\overline{\Omega}\subset \Omega^o$. Then the map from $C^{0,\alpha}(\partial\Omega^o)_0 \times C^{0,\alpha}(\partial\Omega) \times C^{0,\alpha}(\partial\Omega)_0 \times \R^2$
	to $C^{1,\alpha}_{\mathrm{h}}(\overline{\Omega^o} \setminus \Omega) \times C^{1,\alpha}_{\mathrm{h}}(\overline{\Omega})$   that  takes a quintuple $(\mu^o,\mu^i,\eta^i,\rho^o,\rho^i)$ to the pair of functions $(U^o_\Omega[\mu^o,\mu^i,\eta^i,\rho^o,\rho^i], U^i_\Omega[\mu^o,\mu^i,\eta^i,\rho^o,\rho^i])$ defined by
	\begin{equation}\label{U^o,U^i}
	\begin{aligned}
	& U^o_\Omega[\mu^o,\mu^i,\eta^i,\rho^o,\rho^i] \equiv (v^+_{\Omega^o} [\mu^o] + v^-_{\Omega}[\mu^i] + \rho^o)_{| \overline{\Omega^o} \setminus \Omega}
	\\
	& U^i_\Omega[\mu^o,\mu^i,\eta^i,\rho^o,\rho^i] \equiv v^+_{\Omega}[\eta^i] + \rho^i
	\end{aligned}
	\end{equation}
	is bijective.
\end{lemma}

\begin{proof} \ 
	The map is well defined. Indeed, by the harmonicity and regularity properties of single layer potentials (cf.~Theorem \ref{sdp} (i)-(ii)), we know that
	\begin{align*}
	&\Delta U^o_\Omega[\mu^o,\mu^i,\eta^i,\rho^o,\rho^i]=0 \quad \text{on } \Omega^o \setminus \overline{\Omega},
	\\
	&\Delta U^i_\Omega[\mu^o,\mu^i,\eta^i,\rho^o,\rho^i]=0 \quad \text{on } \Omega,
	\\
	&(U^o_\Omega[\mu^o,\mu^i,\eta^i,\rho^o,\rho^i],U^i_\Omega[\mu^o,\mu^i,\eta^i,\rho^o,\rho^i]) \in C^{1,\alpha}(\overline{\Omega^o} \setminus \Omega) \times C^{1,\alpha}(\overline{\Omega}),
	\end{align*}
	for all $(\mu^o,\mu^i,\eta^i,\rho^o,\rho^i) \in C^{0,\alpha}(\partial\Omega^o)_0 \times C^{0,\alpha}(\partial\Omega) \times C^{0,\alpha}(\partial\Omega)_0 \times \R^2$. We now show that it is bijective. So, we take a pair of functions $(h^o,h^i) \in C^{1,\alpha}_{\mathrm{h}}(\overline{\Omega^o} \setminus \Omega) \times C^{1,\alpha}_{\mathrm{h}}(\overline{\Omega})$ and we prove that there exists a unique quintuple $(\mu^o,\mu^i,\eta^i,\rho^o,\rho^i) \in C^{0,\alpha}(\partial\Omega^o)_0 \times C^{0,\alpha}(\partial\Omega) \times C^{0,\alpha}(\partial\Omega)_0 \times \R^2$ such that 
	\begin{equation}\label{h^o,h^i}
	(U^o_\Omega[\mu^o,\mu^i,\eta^i,\rho^o,\rho^i],U^i_\Omega[\mu^o,\mu^i,\eta^i,\rho^o,\rho^i])=(h^o,h^i).
	\end{equation}
	By the uniqueness of the classical solution of the Dirichlet boundary value problem, the second equation in \eqref{h^o,h^i} is equivalent to
	\begin{equation}\label{h^i}
	V_{\partial\Omega}[\eta^i] + \rho^i = h^i_{| \partial \Omega}
	\end{equation}
	(notice that, since $h^i$ is an element of $C^{1,\alpha}_{\mathrm{h}}(\overline{\Omega})$, we have $h^i_{|\partial\Omega} \in C^{1,\alpha} (\partial\Omega) \subseteq C^{0,\alpha} (\partial\Omega)$).
	By Theorem \ref{sdp} (iv), there exists a unique pair $(\eta^i,\rho^i) \in C^{0,\alpha}(\partial\Omega)_0 \times \R$ such that \eqref{h^i} holds.
	Then it remains to show that there exists a unique triple $(\mu^o,\mu^i,\rho^o) \in C^{0,\alpha}(\partial\Omega^o)_0 \times C^{0,\alpha}(\partial\Omega) \times \R$ such that 
	\begin{equation}\label{h^o}
		(v^+_{\Omega^o} [\mu^o] + v^-_{\Omega^i}[\mu^i] + \rho^o)_{| \overline{\Omega^o} \setminus \Omega} = h^o.
	\end{equation}
	By the jump relations for the single layer potential (cf.~Theorem \ref{sdp} (iii)) and by the uniqueness of the classical solution of the Neumann-Dirichlet mixed boundary value problem, 
	 equation \eqref{h^o} is equivalent to the following system of integral equations:
	\begin{equation}\label{sysinteq h^o}
	\begin{aligned}
	& V_{\partial\Omega^o} [\mu^o] + v^-_{\Omega}[\mu^i]_{|\partial\Omega^o} + \rho^o = h^o_{|\partial\Omega^o},
	\\
	& \left( \frac{1}{2} I + W^\ast_{\partial\Omega} \right) [\mu^i] + \nu_{\Omega} \cdot \nabla v^+_{\Omega^o}[\mu^o]_{|\partial\Omega} = \nu_{\Omega} \cdot \nabla h^o_{|\partial\Omega}
	\end{aligned}
	\end{equation}
	(notice that, by $h^o \in C^{1,\alpha}_{\mathrm{h}}(\overline{\Omega})$, we get $h^o_{|\partial\Omega} \in C^{1,\alpha} (\partial\Omega) \subseteq C^{0,\alpha} (\partial\Omega)$ and $\nu_{\Omega} \cdot \nabla h^o_{|\partial\Omega} \in C^{0,\alpha} (\partial\Omega)$).
	Then we observe that by Theorem \ref{sdp} (iv), the map from $C^{0,\alpha}(\partial\Omega^o)_0 \times C^{0,\alpha}(\partial\Omega) \times \R$ to $C^{0,\alpha}(\partial\Omega^o) \times C^{0,\alpha}(\partial\Omega)$  that  takes a triple $(\mu^o,\mu^i,\rho^o)$ to the pair of functions 
	 $\left(V_{\partial\Omega^o} [\mu^o] + \rho^o, \frac{1}{2} \mu^i \right)$	is an isomorphism. Moreover, by the properties of integral operators with real analytic kernel and no singularities (cf. Lanza de Cristoforis and Musolino \cite[Prop. 4.1]{LaMu13}) and by Theorem \ref{Schaudercompact}, the map from $C^{0,\alpha}(\partial\Omega^o)_0 \times C^{0,\alpha}(\partial\Omega) \times \R$ to $C^{0,\alpha}(\partial\Omega^o) \times C^{0,\alpha}(\partial\Omega)$  that  takes a triple $(\mu^o,\mu^i,\rho^o)$ to the pair of functions $(v^-_{\Omega}[\mu^i]_{|\partial\Omega^o},
	W^\ast_{\partial\Omega}[\mu^i] + \nu_{\Omega} \cdot \nabla v^+_{\Omega^o}[\mu^o]_{|\partial\Omega})$
	is compact. Hence, the map from $C^{0,\alpha}(\partial\Omega^o)_0 \times C^{0,\alpha}(\partial\Omega) \times \R$ to $C^{0,\alpha}(\partial\Omega^o) \times C^{0,\alpha}(\partial\Omega)$  that  takes a triple $(\mu^o,\mu^i,\rho^o)$ to the pair of functions
	\begin{equation*}
	\left(V_{\partial\Omega^o} [\mu^o] + v^-_{\Omega}[\mu^i]_{|\partial\Omega^o} + \rho^o, \left( \frac{1}{2} I + W^\ast_{\partial\Omega} \right) [\mu^i] + \nu_{\Omega} \cdot \nabla v^+_{\Omega^o}[\mu^o]_{|\partial\Omega} \right)
	\end{equation*}
	is a compact perturbation of an isomorphism and therefore it is a Fredholm operator of index 0. Thus, to complete the proof, it suffices to show that \eqref{sysinteq h^o} with $(h^o_{|\partial\Omega^o}, \nu_{\Omega} \cdot \nabla h^o_{|\partial\Omega}) = (0,0)$ implies $(\mu^o,\mu^i,\rho^o)=(0,0,0)$. If 
	\begin{equation}\label{sysinteq h^o=0}
	\left(V_{\partial\Omega^o} [\mu^o] + v^-_{\Omega}[\mu^i]_{|\partial\Omega^o} + \rho^o, \left( \frac{1}{2} I + W^\ast_{\partial\Omega} \right) [\mu^i] + \nu_{\Omega} \cdot \nabla v^+_{\Omega^o}[\mu^o]_{|\partial\Omega} \right) =(0,0),
	\end{equation}
	then by the jump relations for the single layer potential (cf.~Theorem \ref{sdp} (iii)) and by the uniqueness of the classical solution of Neumann-Dirichlet mixed boundary value problem, one deduces that $(v^+_{\Omega^o} [\mu^o] + v^-_{\Omega}[\mu^i] + \rho^o)_{| \overline{\Omega^o} \setminus \Omega}=0$. Moreover, by the continuity of $v_{\Omega}[\mu^i]$ in $\R^n$, we have that $(v^+_{\Omega^o} [\mu^o] + v^-_{\Omega}[\mu^i] + \rho^o)_{| \partial\Omega} = (v^+_{\Omega^o} [\mu^o] + v^+_{\Omega}[\mu^i] + \rho^o)_{|\partial\Omega} = 0$. Then by the uniqueness of the classical solution of Dirichlet boundary value problem in $\Omega$ we deduce that
	\begin{equation}\label{mu^o mu^i rho^o}
	(v^+_{\Omega^o} [\mu^o] + v^+_{\Omega}[\mu^i] + \rho^o)_{| \overline{\Omega} }=0.
	\end{equation} 
	By the jump relations for the single layer potential (cf.~Theorem \ref{sdp} (iii)), adding and subtracting the term $\nu_{\Omega} \cdot \nabla ( v^+_{\Omega^o}[\mu^o] +\rho^o)_{|\partial\Omega}$ and taking into account \eqref{mu^o mu^i rho^o}, we get 
	\begin{equation*}
	\begin{split}
	\mu^i &= \nu_{\Omega} \cdot \nabla v^-_{\Omega}[\mu^i]_{|\partial\Omega} - \nu_{\Omega} \cdot \nabla v^+_{\Omega}[\mu^i]_{|\partial\Omega}
	\\
	&= \nu_{\Omega} \cdot \nabla ( v^+_{\Omega^o}[\mu^o] + v^-_{\Omega}[\mu^i] +\rho^o)_{|\partial\Omega} -  \nu_{\Omega} \cdot \nabla ( v^+_{\Omega^o}[\mu^o] + v^+_{\Omega}[\mu^i] +\rho^o)_{|\partial\Omega} = 0.
	\end{split}
	\end{equation*}
	Thus, by \eqref{sysinteq h^o=0}, we obtain $V_{\Omega^o} [\mu^o] + \rho^o = 0$ on $\partial \Omega^o$, which implies $(\mu^o,\rho^o)=(0,0)$ (cf.~Theorem \ref{sdp} (iv)). Hence $(\mu^o,\mu^i,\rho^o)=(0,0,0)$ and the proof is complete.
\end{proof}

To represent the boundary conditions of a linearised version of problem \eqref{princeq}, we find convenient to introduce a matrix function 
\[
A(\cdot) = \begin{pmatrix}
A_{11}(\cdot) & A_{12}(\cdot)
\\
A_{21}(\cdot) & A_{22}(\cdot)
\end{pmatrix} : \partial \Omega^i \to M_2(\R).
\] 
Here above, the symbol $M_2(\R)$ denotes the set of $2\times 2$ matrices with real entries. We set
\[
\tilde{A}(\cdot)\equiv \begin{pmatrix}
A_{11} (\cdot)& A_{12}(\cdot) \\
-A_{21} (\cdot)& -A_{22} (\cdot)
\end{pmatrix}\, .
\] 
We will assume the following conditions on the matrix $A$:
\begin{equation}\label{Acondition}
\begin{split}
&\bullet \, A_{j,k} \in C^{0,\alpha} (\partial\Omega^i) \mbox{ for all } j,k \in \{1,2\};
\\
&\bullet \,\mbox{For every } (\xi_1,\xi_2) \in \R^2, (\xi_1,\xi_2) \tilde{A} (\xi_1,\xi_2)^T \geq 0 \mbox{ on } \partial\Omega^i;
\\
&\bullet \, \mbox{If } (c_1,c_2) \in \R^2 \mbox{ and } A(x)(c_1,c_2)^T = 0 \mbox{ for all } x \in \partial \Omega^i, \mbox{ then } (c_1,c_2)=(0,0). 
\end{split}
\end{equation}
We remark that in literature the third condition in \eqref{Acondition} is often replaced by a condition on the invertibility of the matrix $A$, namely
\begin{equation}\label{A*condition}
\bullet \mbox{There exists a point } x \in \partial\Omega^i \mbox{ such that } A(x) \mbox{ is invertible}.
\end{equation}
We point out that, for instance, the matrix $A(x) = \begin{pmatrix}
x_1^2 &  x_1\\
-x_1 & -1 
\end{pmatrix}$ with $x = (x_1,\dots,x_n) \in \partial\Omega^i$ satisfies the third condition in \eqref{Acondition} but not condition \eqref{A*condition}. Then by a standard energy argument we deduce the following result on the uniqueness of the solution of a transmission problem.

\begin{lemma}\label{Alemma}
	Let $A$ be as in \eqref{Acondition}. Then the unique solution in $C^{1,\alpha}(\overline{\Omega^o} \setminus \Omega^i) \times C^{1,\alpha}(\overline{\Omega^i})$ of problem
	\begin{equation}\label{Aproblem}
	\begin{cases}
	\Delta u^o = 0 & \quad \mbox{in } \Omega^o \setminus \overline{\Omega^i}, 
	\\
	\Delta u^i = 0 & \quad \mbox{in } \Omega^i, 
	\\
	\nu_{\Omega^o}(x) \cdot \nabla u^o(x)= 0 & \quad\forall x \in \partial \Omega^o, 
	\\
	\nu_{\Omega^i}(x) \cdot \nabla u^o (x) - A_{11}(x) u^o(x) - A_{12}(x) u^i(x) = 0 & \quad\forall x \in \partial \Omega^i,
	\\
	\nu_{\Omega^i}(x) \cdot \nabla u^i (x) - A_{21}(x) u^o(x) - A_{22}(x) u^i(x) = 0 & \quad \forall x \in \partial \Omega^i,
	\end{cases}
	\end{equation}
	is $(u^o,u^i)=(0,0)$.
\end{lemma}

In the following proposition, we investigate the properties of an auxiliary boundary operator, $J_A$, which we will exploit in the integral formulation of our problem, in order to recast a fixed point equation. More precisely, we prove that $J_A$ is an isomorphism in $L^2$, in $C^0$, and in $C^{0,\alpha}$. All those three frameworks  will be important: the first setting is suitable to use Fredholm theory and to directly prove the isomorphic property of $J_A$, the second setting will be used in order to apply Leray-Shauder Theorem to the aforementioned obtained fixed point equation (see Propositions \ref{Tcontcomp} and \ref{prop mu_0} below) and the third setting will be central to deduce that the solution of problem \eqref{princeq} we built is actually a classical solution, in particular of class $C^{1,\alpha}$ (cf. Propositions \ref{propintsys} and \ref{prop u^o_0,u^i_0})

\begin{prop}\label{J_A}
	Let $A$ be as in \eqref{Acondition}. Let $J_A$ be the map from $L^2(\partial\Omega^o)_0 \times L^2(\partial\Omega^i) \times L^2(\partial\Omega^i)_0 \times \R^2$ to $L^2(\partial\Omega^o) \times (L^2(\partial\Omega^i))^2$  that  takes a quintuple $(\mu^o,\mu^i,\eta^i,\rho^o,\rho^i)$ to the triple $J_A[\mu^o,\mu^i,\eta^i,\rho^o,\rho^i]$ defined by
	\begin{equation}\label{J_A eq}
	\begin{aligned}
	J_{A,1}[\mu^o,\mu^i,\eta^i,\rho^o,\rho^i] &\equiv \left( -\frac{1}{2} I + W^\ast_{\partial\Omega^o} \right) [\mu^o] + \nu_{\Omega^o} \cdot \nabla  v^-_{\Omega^i}[\mu^i]_{|\partial\Omega^o} \qquad \mbox{on } \partial\Omega^o,
	\\
	J_{A,2}[\mu^o,\mu^i,\eta^i,\rho^o,\rho^i] &\equiv \left( \frac{1}{2} I + W^\ast_{\partial\Omega^i} \right) [\mu^i] + \nu_{\Omega^i} \cdot \nabla  v^+_{\Omega^o}[\mu^o]_{|\partial\Omega^i} 
	\\
	 \quad - (A_{11},A_{12}) &\cdot (v^+_{\Omega^o}[\mu^o]_{|\partial\Omega^i} + V_{\partial\Omega^i}[\mu^i] + \rho^o ,V_{\partial\Omega^i}[\eta^i] + \rho^i  ) \qquad\mbox{on } \partial\Omega^i,
	\\
	J_{A,3}[\mu^o,\mu^i,\eta^i,\rho^o,\rho^i] &\equiv \left( -\frac{1}{2} I + W^\ast_{\partial\Omega^i} \right) [\eta^i] 
	\\
 \quad - (A_{21},A_{22}) &\cdot (v^+_{\Omega^o}[\mu^o]_{|\partial\Omega^i} + V_{\partial\Omega^i}[\mu^i] + \rho^o ,V_{\partial\Omega^i}[\eta^i] + \rho^i  )  \qquad \mbox{on } \partial\Omega^i.
	\end{aligned}
	\end{equation}
	Then the following statements hold.
	\begin{enumerate}
		\item[(i)] $J_A$ is a linear isomorphism from $L^2(\partial\Omega^o)_0 \times L^2(\partial\Omega^i) \times L^2(\partial\Omega^i)_0 \times \R^2$ to $L^2(\partial\Omega^o) \times (L^2(\partial\Omega^i))^2$.
		
		\item[(ii)] $J_A$ is a linear isomorphism from $C^0(\partial\Omega^o)_0 \times C^0(\partial\Omega^i) \times C^0(\partial\Omega^i)_0 \times \R^2$ to $C^0(\partial\Omega^o) \times (C^0(\partial\Omega^i))^2$.
		
		\item[(iii)] $J_A$ is a linear isomorphism from $C^{0,\alpha}(\partial\Omega^o)_0 \times C^{0,\alpha}(\partial\Omega^i) \times C^{0,\alpha}(\partial\Omega^i)_0 \times \R^2$ to $C^{0,\alpha}(\partial\Omega^o) \times (C^{0,\alpha}(\partial\Omega^i))^2$.
	\end{enumerate}
\end{prop}

\begin{proof} \ 
	We first prove (i). We write $J_A$ in the form $J_A = \tilde{J}^{+}_A \circ \tilde{J}_A \circ \tilde{J}^{-}_A$, 	where $\tilde{J}^{-}_A$ is the inclusion of $L^2(\partial\Omega^o)_0 \times L^2(\partial\Omega^i) \times L^2(\partial\Omega^i)_0 \times \R^2$ into $L^2(\partial\Omega^o) \times (L^2(\partial\Omega^i))^2 \times \R^2$, $\tilde{J}_A$ is the map from $L^2(\partial\Omega^o) \times (L^2(\partial\Omega^i))^2 \times \R^2$ into itself  that  takes $(\mu^o,\mu^i,\eta^i,\rho^o,\rho^i)$ to the quintuple $\tilde{J}_A[\mu^o,\mu^i,\eta^i,\rho^o,\rho^i]$ defined by
	\begin{align*}
	\tilde{J}_{A,1}[\mu^o,\mu^i,\eta^i,\rho^o,\rho^i] &\equiv \left( -\frac{1}{2} I + W^\ast_{\partial\Omega^o} \right) [\mu^o] + \nu_{\Omega^o} \cdot \nabla  v^-_{\Omega^i}[\mu^i]_{|\partial\Omega^o} \qquad \mbox{on } \partial\Omega^o,
	\\
	\tilde{J}_{A,2}[\mu^o,\mu^i,\eta^i,\rho^o,\rho^i] &\equiv \left( \frac{1}{2} I + W^\ast_{\partial\Omega^i} \right) [\mu^i] + \nu_{\Omega^i} \cdot \nabla  v^+_{\Omega^o}[\mu^o]_{|\partial\Omega^i} 
	\\
	 - (A_{11},A_{12}) &\cdot (v^+_{\Omega^o}[\mu^o]_{|\partial\Omega^i} + V_{\partial\Omega^i}[\mu^i] ,V_{\partial\Omega^i}[\eta^i] ) \qquad \mbox{on } \partial\Omega^i,
	\\
    \tilde{J}_{A,3}[\mu^o,\mu^i,\eta^i,\rho^o,\rho^i] &\equiv \left( -\frac{1}{2} I + W^\ast_{\partial\Omega^i} \right) [\eta^i]\\ -  (A_{21},A_{22})& \cdot (v^+_{\Omega^o}[\mu^o]_{|\partial\Omega^i} + V_{\partial\Omega^i}[\mu^i],V_{\partial\Omega^i}[\eta^i]) \qquad \mbox{on } \partial\Omega^i,
    \\
    \tilde{J}_{A,4}[\mu^o,\mu^i,\eta^i,\rho^o,\rho^i] &\equiv \rho^o,
    \\
    \tilde{J}_{A,5}[\mu^o,\mu^i,\eta^i,\rho^o,\rho^i] &\equiv \rho^i,
	\end{align*} 
	and $\tilde{J}^{+}_A$ is the map from $L^2(\partial\Omega^o) \times (L^2(\partial\Omega^i))^2 \times \R^2$ into $L^2(\partial\Omega^o) \times (L^2(\partial\Omega^i))^2$  that  takes a quintuple $(f,g_1,g_2,c_1,c_2)$ to the triple $\tilde{J}^{+}_A[f,g_1,g_2,c_1,c_2]$ defined by
	\begin{equation*}
	\tilde{J}^{+}_A[f,g_1,g_2,c_1,c_2] \equiv 
	(f, g_1 - (A_{11},A_{12}) \cdot (c_1,c_2),g_2 - (A_{21},A_{22}) \cdot (c_1,c_2)  ) .
	\end{equation*}
	Then we observe that $\tilde{J}^{+}_A$ is a Fredholm operator of index $2$, because $\mathrm{Coker}\, \tilde{J}^{+}_A = \{0\}$ and $\mathrm{Ker}\, \tilde{J}^{+}_A = \mathrm{Span}\, \{(0,A_{11},A_{21},1,0), (0,A_{12},A_{22},0,1)\}$, and that $\tilde{J}^{-}_A$ is Fredholm of index $-2$, because $\mathrm{Ker}\, \tilde{J}^{-}_A = \{0\}$ and $\mathrm{Coker}\, \tilde{J}^{-}_A = \mathrm{Span}\, \{(1,0,0,0,0), (0,0,1,0,0)\}$. Next, we observe that the map from $L^2(\partial\Omega^o) \times (L^2(\partial\Omega^i))^2 \times \R^2$ into itself  that  takes a quintuple $(\mu^o,\mu^i,\eta^i,\rho^o,\rho^i)$ to the quintuple $(-\frac{1}{2}\mu^o,\frac{1}{2}\mu^i,-\frac{1}{2}\eta^i,\rho^o,\rho^i)$ is  a linear isomorphism. Moreover, by the mapping properties of the integral operators with real analytic kernel and no singularity (cf. Lanza de Cristoforis and Musolino \cite[Prop. 4.1]{LaMu13}), by the compactness of the operators $W^\ast_{\partial\Omega^o}$ and $W^\ast_{\partial\Omega^i}$ from $L^2(\partial\Omega^o)$ to itself and from $L^2(\partial\Omega^i)$ to itself, respectively (see comments below Definition \ref{defV-W}), by the compactness of the operator $V_{\partial\Omega^i}$ from $L^2(\partial\Omega^i)$ into itself (see Costabel \cite[Thm. 1]{Co88}), and by the bilinearity and continuity of the product from $C^{0,\alpha}(\partial\Omega^i) \times L^2(\partial\Omega^i)$ to $L^2(\partial\Omega^i)$, we deduce that the map from $L^2(\partial\Omega^o) \times (L^2(\partial\Omega^i))^2 \times \R^2$ into itself  that  takes a quintuple $(\mu^o,\mu^i,\eta^i,\rho^o,\rho^i)$ to  the quintuple $\tilde{J}^C_A[\mu^o,\mu^i,\eta^i,\rho^o,\rho^i]$ defined by
	\begin{align*}
	\tilde{J}^C_{A,1}[\mu^o,\mu^i,\eta^i,\rho^o,\rho^i] &= W^\ast_{\partial\Omega^o} [\mu^o] + \nu_{\Omega^o} \cdot \nabla  v^-_{\Omega^i}[\mu^i]_{|\partial\Omega^o} \qquad \mbox{on } \partial\Omega^o,
	\\
	\tilde{J}^C_{A,2}[\mu^o,\mu^i,\eta^i,\rho^o,\rho^i] & = W^\ast_{\partial\Omega^i} [\mu^i] + \nu_{\Omega^i} \cdot \nabla  v^+_{\Omega^o}[\mu^o]_{|\partial\Omega^i} 
	\\
	 - (A_{11},A_{12}) & \cdot (v^+_{\Omega^o}[\mu^o]_{|\partial\Omega^i} + V_{\partial\Omega^i}[\mu^i] ,V_{\partial\Omega^i}[\eta^i] )
	\qquad \mbox{on } \partial\Omega^i,
	\\
	\tilde{J}^C_{A,3}[\mu^o,\mu^i,\eta^i,\rho^o,\rho^i] & = W^\ast_{\partial\Omega^i} [\eta^i] \\-  (A_{21},A_{22})&  \cdot (v^+_{\Omega^o}[\mu^o]_{|\partial\Omega^i} + V_{\partial\Omega^i}[\mu^i],V_{\partial\Omega^i}[\eta^i])
 \qquad \mbox{on } \partial\Omega^i,
	\\
	\tilde{J}^C_{A,4}[\mu^o,\mu^i,\eta^i,\rho^o,\rho^i] & = 0,
	\\
	\tilde{J}^C_{A,5}[\mu^o,\mu^i,\eta^i,\rho^o,\rho^i] & = 0,
	\end{align*} 
	is compact.	Hence, we conclude that $\tilde{J}_A$ is a compact perturbation of an isomorphism and therefore it is Fredholm of index 0. Since the index of a composition of Fredholm operators is the sum of the indexes of the components, we deduce that $J_A$ is a Fredholm operator of index $0$. Therefore, in order to complete the proof of point $(i)$, it suffices to prove that $J_A$ is injective. Thus, we now assume that $(\mu^o,\mu^i,\eta^i,\rho^o,\rho^i) \in L^2(\partial\Omega^o)_0 \times L^2(\partial\Omega^i) \times L^2(\partial\Omega^i)_0 \times \R^2$ and that
	\begin{equation}\label{J_A=0}
	J_A[\mu^o,\mu^i,\eta^i,\rho^o,\rho^i] = (0,0,0).
	\end{equation}
	We first verify that  $(\mu^o,\mu^i,\eta^i)$ is actually in $C^0(\partial\Omega^o)_0 \times C^0(\partial\Omega^i) \times C^0(\partial\Omega^i)_0$. In fact,  by the definition of $J_{A,1}$ in \eqref{J_A eq}, by the fact that $\nu_{\Omega^o} \cdot \nabla  v^-_{\Omega^i}[\mu^i]_{|\partial\Omega^o} \in C^{0}(\partial\Omega^o)$ (cf. Lanza de Cristoforis and Musolino \cite[Prop. 4.1]{LaMu13}), and by Theorem \ref{regularitytheorem}, we deduce that $\mu^o \in C^{0}(\partial\Omega^o)$. Let $\mathcal{K}\equiv(\mathcal{K}_1, \mathcal{K}_2)$ be the map from $L^2(\partial\Omega^i)\times L^2(\partial\Omega^i)_0$ to itself  that takes a pair $(\mu^i,\eta^i)\in L^2(\partial\Omega^i)\times L^2(\partial\Omega^i)_0$ to
	\begin{align*}
	\mathcal{K}_1[\mu^i,\eta^i] &\equiv 2W^\ast_{\partial\Omega^i}[\mu^i] - 2A_{11} V_{\partial\Omega^i}[\mu^i] - 2A_{12}V_{\partial\Omega^i}[\eta^i] &&\mbox{on } \partial\Omega^i,
	\\
	\mathcal{K}_2[\mu^i,\eta^i] &\equiv -2W^\ast_{\partial\Omega^i} [\eta^i] 
	+2 A_{21} V_{\partial\Omega^i}[\mu^i] +2 A_{22} V_{\partial\Omega^i}[\eta^i]   &&\mbox{on } \partial\Omega^i.
	\end{align*}
	Notice that each component of $\mathcal{K}$ is a linear combinations of integral operators  with a continuous kernel of order $n-2$ (see Definition \ref{kerneldef} and comments below). By the fact that $\mu^o \in C^0(\partial\Omega^i)$ and by the first condition in \eqref{Acondition}, we know that  
	$2(A_{11},A_{12}) \cdot (v^+_{\Omega^o}[\mu^o]_{|\partial\Omega^i} + \rho^o , \rho^i  ), -2(A_{21},A_{22}) \cdot (v^+_{\Omega^o}[\mu^o]_{|\partial\Omega^i} + \rho^o ,\rho^i  )$ belong to $C^{0,\alpha}(\partial\Omega^i) \subseteq C^{0}(\partial\Omega^i)$. 
	Then \eqref{J_A=0} and the definition of the operator $\mathcal{K}$ imply that   
$(I + \mathcal{K})[\mu^i,\eta^i] \in (C^{0}(\partial\Omega^i))^2$.	Hence, by  Theorem  \ref{regularityvecttheorem} we conclude that $(\mu^i,\eta^i) \in  C^0(\partial\Omega^i) \times C^0(\partial\Omega^i)_0$. Then by mapping properties of integral operators with real analytic kernel and no singularity (cf. Lanza de Cristoforis and Musolino \cite[Prop. 4.1]{LaMu13}) and by classical known results in potential theory (cf.~Miranda \cite[Chap. II, \S 14]{Mi70}), we know that
	 $ \nu_{\Omega^o} \cdot \nabla  v^-_{\Omega^i}[\mu^i]_{|\partial\Omega^o} \in C^{0,\alpha}(\partial\Omega^o)$ and	$v^+_{\Omega^o}[\mu^o]_{|\partial\Omega^i}, \, \nu_{\Omega^i} \cdot \nabla  v^+_{\Omega^o}[\mu^o]_{|\partial\Omega^i}, \, V_{\partial\Omega^i}[\eta^i], \,
	V_{\partial\Omega^i}[\mu^i]\in C^{0,\alpha}(\partial\Omega^i)$.	Hence, by \eqref{J_A=0} and by the membership of $A \in  M_2(C^{0,\alpha} (\partial\Omega^i))$ (cf.~first condition in \eqref{Acondition}), we obtain that $\left( -\frac{1}{2} I + W^\ast_{\partial\Omega^o} \right) [\mu^o] \in C^{0,\alpha}(\partial\Omega^o)$	and
$	 \left( \frac{1}{2} I + W^\ast_{\partial\Omega^i} \right) [\mu^i], \left( -\frac{1}{2} I + W^\ast_{\partial\Omega^i} \right) [\eta^i] \in C^{0,\alpha}(\partial\Omega^i)$.	Then Theorem \ref{regularitytheorem} implies $(\mu^o,\mu^i,\eta^i) \in C^{0,\alpha}(\partial\Omega^o)_0 \times C^{0,\alpha}(\partial\Omega^i) \times C^{0,\alpha}(\partial\Omega^i)_0$. By the jump relations (cf.~Theorem \ref{sdp} (iii)), by Lemma  \ref{rapprharm}, and by \eqref{J_A=0}, we deduce that the pair  $(U^o_{\Omega^i}[\mu^o,\mu^i,\eta^i,\rho^o,\rho^i],U^i_{\Omega^i}[\mu^o,\mu^i,\eta^i,\rho^o,\rho^i])$ defined by \eqref{U^o,U^i} is a solution of the boundary value problem \eqref{Aproblem}.
Then by Lemma \ref{Alemma}, we have that  
\[
(U^o_{\Omega^i}[\mu^o,\mu^i,\eta^i,\rho^o,\rho^i],U^i_{\Omega^i}[\mu^o,\mu^i,\eta^i,\rho^o,\rho^i])=(0,0),\]
which implies $(\mu^o,\mu^i,\eta^i,\rho^o,\rho^i)=(0,0,0,0,0)$, by the uniqueness of the representation provided by Lemma \ref{rapprharm}. We now prove statement (ii). First we note that  the integral operators that appear in the definition \eqref{J_A eq} of $J_A$  have either a weakly singular or a real analytic kernel. It follows that
 $J_A$ is continuous from $C^0(\partial\Omega^o)_0 \times C^0(\partial\Omega^i) \times C^0(\partial\Omega^i)_0 \times \R^2$ to $C^0(\partial\Omega^o) \times (C^0(\partial\Omega^i))^2$ (cf.~Lanza de Cristoforis and Musolino \cite[Prop. 4.1]{LaMu13}  for the properties of integral operators with real analytic kernels). Then we observe that,  by Theorems  \ref{regularityvecttheorem} and \ref{regularitytheorem}, if we have $	J_A[\mu^o,\mu^i,\eta^i,\rho^o,\rho^i] \in C^0(\partial\Omega^o) \times (C^0(\partial\Omega^i))^2$		for some $(\mu^o,\mu^i,\eta^i,\rho^o,\rho^i) \in L^2(\partial\Omega^o)_0 \times L^2(\partial\Omega^i) \times L^2(\partial\Omega^i)_0 \times \R^2$,
		then $(\mu^o,\mu^i,\eta^i) \in C^0(\partial\Omega^o)_0 \times C^0(\partial\Omega^i) \times C^0(\partial\Omega^i)_0$ (see also the argument used after \eqref{J_A=0} to prove that $(\mu^o,\mu^i,\eta^i)$ belongs to $C^0(\partial\Omega^o)_0 \times C^0(\partial\Omega^i) \times C^0(\partial\Omega^i)_0$). Then, by statement (i) we deduce that $J_A$ is a bijective continuous linear map from $C^0(\partial\Omega^o)_0 \times C^0(\partial\Omega^i) \times C^0(\partial\Omega^i)_0 \times \R^2$ to $C^0(\partial\Omega^o) \times (C^0(\partial\Omega^i))^2$. By the Open Mapping Theorem it follows that $J_A$ is a linear homeomorphism from $C^0(\partial\Omega^o)_0 \times C^0(\partial\Omega^i) \times C^0(\partial\Omega^i)_0 \times \R^2$ to $C^0(\partial\Omega^o) \times (C^0(\partial\Omega^i))^2$.
		The proof of statement (iii) is similar to that of statement (ii) and we leave it to the zealous reader (see also the argument used after \eqref{J_A=0} to prove that $(\mu^o,\mu^i,\eta^i)$ belongs to $C^{0,\alpha}(\partial\Omega^o)_0 \times C^{0,\alpha}(\partial\Omega^i) \times C^{0,\alpha}(\partial\Omega^i)_0$).
\end{proof}
We are now ready to convert  \eqref{princeq} into a system of integral equations. 
\begin{prop}\label{propintsys}
	Let $A$ be as in \eqref{Acondition}. Let $(\mu^o,\mu^i,\eta^i,\rho^o,\rho^i) \in C^{0,\alpha}(\partial\Omega^o)_0 \times C^{0,\alpha}(\partial\Omega^i) \times C^{0,\alpha}(\partial\Omega^i)_0 \times \R^2$. Let $(U^o_{\Omega^i}[\cdot,\cdot,\cdot,\cdot,\cdot],U^i_{\Omega^i}[\cdot,\cdot,\cdot,\cdot,\cdot])$ be defined by \eqref{U^o,U^i}. Let $J_A$ be as in Proposition \ref{J_A}.
	Then $(U^o_{\Omega^i}[\mu^o,\mu^i,\eta^i,\rho^o,\rho^i],U^i_{\Omega^i}[\mu^o,\mu^i,\eta^i,\rho^o,\rho^i])$  is a solution of  \eqref{princeq} if and only if 
	\begin{equation}\label{princintsys}
	\begin{aligned}
	\begin{pmatrix}
	\mu^o
	\\
	\mu^i
	\\
	\eta^i
	\\
	\rho^o
	\\ 
	\rho^i
	\end{pmatrix}
	= J_A^{(-1)}&
	\left[
	\begin{pmatrix}
	f^o
	\\
	\mathcal{N}_{F_1}(v^+_{\Omega^o}[\mu^o]_{|\partial\Omega^i} + V_{\partial\Omega^i}[\mu^i] +\rho^o ,V_{\partial\Omega^i}[\eta^i] +\rho^i)
	\\
	\mathcal{N}_{F_2}(v^+_{\Omega^o}[\mu^o]_{|\partial\Omega^i} + V_{\partial\Omega^i}[\mu^i] +\rho^o ,V_{\partial\Omega^i}[\eta^i] +\rho^i)
	\end{pmatrix}
	\right. 
	\\
	& \left.
	- \begin{pmatrix}
	0 & 0 & 0 \\
	0 & A_{11} & A_{12} \\
	0 & A_{21} & A_{22} 
	\end{pmatrix} 
	\begin{pmatrix}
	0
	\\
	v^+_{\Omega^o}[\mu^o]_{|\partial\Omega^i} + V_{\partial\Omega^i}[\mu^i] +\rho^o
	\\
	V_{\partial\Omega^i}[\eta^i] +\rho^i
	\end{pmatrix} \right].
	\end{aligned}
	\end{equation}
\end{prop}

\begin{proof} \ 
	By Lemma \ref{rapprharm} and by the jump relations of Theorem \ref{sdp}, we know that if $(\mu^o,\mu^i,\eta^i,\rho^o,\rho^i) \in C^{0,\alpha}(\partial\Omega^o)_0 \times C^{0,\alpha}(\partial\Omega^i) \times C^{0,\alpha}(\partial\Omega^i)_0 \times \R^2$ then the pair \[(U^o_{\Omega^i}[\mu^o,\mu^i,\eta^i,\rho^o,\rho^i],U^i_{\Omega^i}[\mu^o,\mu^i,\eta^i,\rho^o,\rho^i])\]
	defined by \eqref{U^o,U^i} is a solution of problem \eqref{princeq} if an only if
	\begin{equation}\label{intsys}
	\begin{split}
	& \begin{pmatrix}
	\left( -\frac{1}{2} I + W^\ast_{\partial\Omega^o} \right) [\mu^o] + \nu_{\Omega^o} \cdot \nabla  v^-_{\Omega^i}[\mu^i]_{|\partial\Omega^o}
	\\
	\left( \frac{1}{2} I + W^\ast_{\partial\Omega^i} \right) [\mu^i] + \nu_{\Omega^i} \cdot \nabla  v^+_{\Omega^o}[\mu^o]_{|\partial\Omega^i} 
	\\
	\left( -\frac{1}{2} I + W^\ast_{\partial\Omega^i} \right) [\eta^i] 
	\end{pmatrix}
	\\ &\qquad =
	\begin{pmatrix}
	f^o
	\\
	\mathcal{N}_{F_1}(v^+_{\Omega^o}[\mu^o]_{|\partial\Omega^i} + V_{\partial\Omega^i}[\mu^i] +\rho^o ,V_{\partial\Omega^i}[\eta^i] +\rho^i)
	\\
	\mathcal{N}_{F_2}(v^+_{\Omega^o}[\mu^o]_{|\partial\Omega^i} + V_{\partial\Omega^i}[\mu^i] +\rho^o ,V_{\partial\Omega^i}[\eta^i] +\rho^i)
	\end{pmatrix}.
	\end{split}
	\end{equation}
	Then, by subtracting in both sides of \eqref{intsys} the term
	\begin{equation*}
	\begin{pmatrix}
	0 & 0 & 0 \\
	0 & A_{11} & A_{12} \\
	0 & A_{21} & A_{22} 
	\end{pmatrix} 
	\begin{pmatrix}
	0
	\\
	v^+_{\Omega^o}[\mu^o]_{|\partial\Omega^i} + V_{\partial\Omega^i}[\mu^i] +\rho^o
	\\
	V_{\partial\Omega^i}[\eta^i] +\rho^i
	\end{pmatrix}
	\in C^{0,\alpha}(\partial\Omega^o) \times (C^{0,\alpha}(\partial\Omega^i))^2
	\end{equation*} 
	and by the invertibility of $J_A$ from $C^{0,\alpha}(\partial\Omega^o)_0 \times C^{0,\alpha}(\partial\Omega^i) \times C^{0,\alpha}(\partial\Omega^i)_0 \times \R^2$ to $C^{0,\alpha}(\partial\Omega^o) \times (C^{0,\alpha}(\partial\Omega^i))^2$ provided by Proposition \ref{J_A} (iii), the validity of the statement follows. 
\end{proof}

We now introduce an auxiliary map. If $A$ is as in \eqref{Acondition} and  $J_A$ is as in Proposition \ref{J_A}, we denote by $T_A$ the map from $C^0(\partial\Omega^o)_0 \times C^0(\partial\Omega^i) \times C^0(\partial\Omega^i)_0 \times \R^2$ to $C^0(\partial\Omega^o) \times (C^0(\partial\Omega^i))^2$ defined by
\begin{equation}\label{T}
\begin{aligned}
T_A&(\mu^o,\mu^i,\eta^i,\rho^o,\rho^i) \\&\equiv J_A^{(-1)}
\left[
\begin{pmatrix}
f^o
\\
\mathcal{N}_{F_1}(v^+_{\Omega^o}[\mu^o]_{|\partial\Omega^i} + V_{\partial\Omega^i}[\mu^i] +\rho^o ,V_{\partial\Omega^i}[\eta^i] +\rho^i)
\\
\mathcal{N}_{F_2}(v^+_{\Omega^o}[\mu^o]_{|\partial\Omega^i} + V_{\partial\Omega^i}[\mu^i] +\rho^o ,V_{\partial\Omega^i}[\eta^i] +\rho^i)
\end{pmatrix}
\right. 
\\
& \left.
- \begin{pmatrix}
0 & 0 & 0 \\
0 & A_{11} & A_{12} \\
0 & A_{21} & A_{22} 
\end{pmatrix} 
\begin{pmatrix}
0
\\
v^+_{\Omega^o}[\mu^o]_{|\partial\Omega^i} + V_{\partial\Omega^i}[\mu^i] +\rho^o
\\
V_{\partial\Omega^i}[\eta^i] +\rho^i
\end{pmatrix} \right]\, .
\end{aligned}
\end{equation}

We study the continuity and compactness of $T_A$ in the following proposition.

\begin{prop}\label{Tcontcomp}
	Let $A$ be as in \eqref{Acondition}. Let $T_A$ be as in \eqref{T}. Then $T_A$ is a continuous (nonlinear) operator from $C^0(\partial\Omega^o)_0 \times C^0(\partial\Omega^i) \times C^0(\partial\Omega^i)_0 \times \R^2$ to $C^0(\partial\Omega^o) \times (C^0(\partial\Omega^i))^2$ and  is compact. 
\end{prop}

\begin{proof} \ 
	By the properties of integral operators with real analytic kernel and no singularities (cf. Lanza de Cristoforis and Musolino \cite[Prop. 4.1]{LaMu13}) and by the compactness of the embedding of $C^{0,\alpha}(\partial\Omega^i)$ into $C^0(\partial\Omega^i)$, $v^+_{\Omega^o}[\cdot]_{|\partial\Omega^i}$ is compact from $C^0(\partial\Omega^o)_0$ into $C^0(\partial\Omega^i)$. By mapping properties of the single layer potential (cf.~Miranda \cite[Chap.~II, \S14, III]{Mi70}) and by the compactness of the embedding of $C^{0,\alpha}(\partial\Omega^i)$ into $C^0(\partial\Omega^i)$, $V_{\partial\Omega^i}$ is compact from $C^0(\partial\Omega^i)$ into itself. Hence, by the bilinearity and continuity of the product of continuous functions, the map from $C^0(\partial\Omega^o)_0 \times C^0(\partial\Omega^i) \times C^0(\partial\Omega^i)_0 \times \R^2$ to $C^0(\partial\Omega^o) \times (C^0(\partial\Omega^i))^2$  that  takes the quintuple $(\mu^o,\mu^i,\eta^i,\rho^o,\rho^i)$ to the triple given by
	\begin{equation*}
	\begin{pmatrix}
	0 & 0 & 0 \\
	0 & A_{11} & A_{12} \\
	0 & A_{21} & A_{22} 
	\end{pmatrix} 
	\begin{pmatrix}
	0
	\\
	v^+_{\Omega^o}[\mu^o]_{|\partial\Omega^i} + V_{\partial\Omega^i}[\mu^i] +\rho^o
	\\
	V_{\partial\Omega^i}[\eta^i] +\rho^i
	\end{pmatrix}
	\end{equation*} 
	is continuous and maps bounded sets into sets with compact closure, {i.e.}, is compact. Moreover, by assumption \eqref{introfunconditions}, one readily verifies that the operators $\mathcal{N}_{F_1}$ and $\mathcal{N}_{F_2}$ are continuous from $(C^0(\partial\Omega^i))^2$ into $C^0(\partial\Omega^i)$.  Hence the map from $C^0(\partial\Omega^o)_0 \times C^0(\partial\Omega^i) \times C^0(\partial\Omega^i)_0$ to $C^0(\partial\Omega^o) \times (C^0(\partial\Omega^i))^2$  that  takes the quintuple $(\mu^o,\mu^i,\eta^i,\rho^o,\rho^i)$ to the triple 
	\begin{equation*}
	\begin{pmatrix}
	f^0
	\\
	\mathcal{N}_{F_1}(v^+_{\Omega^o}[\mu^o]_{|\partial\Omega^i} + V_{\partial\Omega^i}[\mu^i] +\rho^o ,V_{\partial\Omega^i}[\eta^i] +\rho^i)
	\\
	\mathcal{N}_{F_2}(v^+_{\Omega^o}[\mu^o]_{|\partial\Omega^i} + V_{\partial\Omega^i}[\mu^i] +\rho^o ,V_{\partial\Omega^i}[\eta^i] +\rho^i)
	\end{pmatrix}
	\end{equation*}
	is compact. Finally, by Proposition \ref{J_A} (ii),  $J_A$ is a linear isomorphism from $C^0(\partial\Omega^o)_0 \times C^0(\partial\Omega^i) \times C^0(\partial\Omega^i)_0 \times \R^2$ to $C^0(\partial\Omega^o) \times (C^0(\partial\Omega^i))^2$ and, accordingly, $T_A$ is compact.
\end{proof}

In what follows we will assume the following growth condition on the pair  $(F_1,F_2)$ with respect to the matrix function $A$  defined as in \eqref{Acondition}:
\begin{equation}\label{conditionF1F2}
\begin{split}
\bullet\, &\mbox{There exist two constants } C_F \in ]0,+\infty[ \mbox{ and } \delta \in ]0,1[ \mbox{ such that}
\\
& 
\qquad
\left|\begin{pmatrix}
F_1(x, \zeta_1,\zeta_2)
\\
F_2(x, \zeta_1,\zeta_2)
\end{pmatrix}
- A(x) 
\begin{pmatrix}
\zeta_1
\\
\zeta_2
\end{pmatrix}
\right|
\leq C_F (1 + |\zeta_1| + |\zeta_2|)^\delta
\\& \mbox{for all } (x,\zeta_1,\zeta_2) \in \partial\Omega^i \times \R^2.
\end{split}
\end{equation}
 In Proposition \ref{prop mu_0} below we prove the existence of a solution in $C^0(\partial\Omega^o)_0 \times C^0(\partial\Omega^i) \times C^0(\partial\Omega^i)_0 \times \R^2$ of \eqref{princintsys}. Our argument exploits the Leray-Schauder Fixed-Point Theorem (cf.~Gilbarg and Trudinger \cite[Thm.~11.3]{GiTr83}). 

\begin{teo}[Leray-Schauder Theorem]\label{Thm Leray Schauder}
	Let $X$ be a Banach space. Let $T$ be a continuous operator from $X$ into itself. If $T$ is compact and there exists a constant $M \in ]0,+\infty[$ such that $\|x\|_{X} \leq M$ for all $(x,\lambda) \in X \times [0,1]$ satisfying $x=\lambda T(x)$, then $T$ has at least one fixed point $x \in X$ such that $\|x\|_{X} \leq M$.
\end{teo}

Then we have the following.

\begin{prop}\label{prop mu_0}
	Let $A$ be as in \eqref{Acondition}. Let assumption \eqref{conditionF1F2} holds. Let $J_A$ be as in Proposition \ref{J_A}. Then the nonlinear system \eqref{princintsys} has at least one solution \[(\mu^o_0,\mu^i_0,\eta^i_0,\rho^o_0,\rho^i_0) \in C^0(\partial\Omega^o)_0 \times C^0(\partial\Omega^i) \times C^0(\partial\Omega^i)_0 \times \R^2.\] 
\end{prop}

\begin{proof} \ 
	We plan to apply the Leray-Schauder Theorem \ref{Thm Leray Schauder} to the operator $T_A$ defined by \eqref{T} in the Banach space $C^0(\partial\Omega^o)_0 \times C^0(\partial\Omega^i) \times C^0(\partial\Omega^i)_0 \times \R^2$. By Proposition \ref{Tcontcomp} we already know that $T_A$ is a continuous operator from $C^0(\partial\Omega^o)_0 \times C^0(\partial\Omega^i) \times C^0(\partial\Omega^i)_0 \times \R^2$ to $C^0(\partial\Omega^o) \times (C^0(\partial\Omega^i))^2$ and maps bounded sets into sets with compact closure. So in order to apply the Leray-Schauder Theorem \ref{Thm Leray Schauder}, we are left to show that if $\lambda \in ]0,1[$ and if
	\begin{equation}\label{mu=lamdaT(mu)}
	(\mu^o,\mu^i,\eta^i,\rho^o,\rho^i) = \lambda T_A(\mu^o,\mu^i,\eta^i,\rho^o,\rho^i)
	\end{equation}
	with $(\mu^o,\mu^i,\eta^i,\rho^o,\rho^i) \in C^0(\partial\Omega^o)_0 \times C^0(\partial\Omega^i) \times C^0(\partial\Omega^i)_0 \times \R^2$,  then there exists a constant $C \in ]0,+\infty[$ (which does not depend on $\lambda$ and $(\mu^o,\mu^i,\eta^i,\rho^o,\rho^i)$), such that
	\begin{equation}\label{mu<C}
	\|(\mu^o,\mu^i,\eta^i,\rho^o,\rho^i)\|_{C^0(\partial\Omega^o) \times (C^0(\partial\Omega^i))^2\times \R^2} \leq C.
	\end{equation}
	By \eqref{mu=lamdaT(mu)} and by $|\lambda|<1$, we readily deduce that
	\begin{equation}\label{|mu|<|T(mu)|}
	\begin{split}
	&\|(\mu^o,\mu^i,\eta^i,\rho^o,\rho^i)\|_{C^0(\partial\Omega^o) \times (C^0(\partial\Omega^i))^2 \times \R^2} \\ &\qquad\leq \|T_A(\mu^o,\mu^i,\eta^i,\rho^o,\rho^i)\|_{C^0(\partial\Omega^o) \times (C^0(\partial\Omega^i))^2}\, .
	\end{split}
	\end{equation}	By the growth condition \eqref{conditionF1F2}, we can show that 
	\begin{equation}\label{inequalityNF1NF2}
	\begin{split}
	&\left\| \begin{pmatrix}
	\mathcal{N}_{F_1}(h^i_1,h^i_2)
	\\
	\mathcal{N}_{F_2}(h^i_1,h^i_2)
	\end{pmatrix} - A \begin{pmatrix}
	h^i_1 \\ h^i_2
	\end{pmatrix} \right\|_{(C^0(\partial\Omega^i))^2} \\
	& \qquad\leq C_F (1+\|h^i_1\|_{C^0(\partial\Omega^i)}+\|h^i_2\|_{C^0(\partial\Omega^i)})^\delta
	\end{split}
	\end{equation}
	for all pair of functions $(h^i_1,h^i_2) \in (C^0(\partial\Omega^i))^2$. Hence, by \eqref{|mu|<|T(mu)|} and by the definition of $T_A$ in \eqref{T}, we deduce that there exist two constants $C_1,C_2 \in ]0,+\infty[$, which depend only on the operator norm of $J_A^{(-1)}$ from $C^0(\partial\Omega^o) \times (C^0(\partial\Omega^i))^2$ to $C^0(\partial\Omega^o)_0 \times C^0(\partial\Omega^i) \times C^0(\partial\Omega^i)_0 \times \R^2$ (cf.~Theorem \ref{J_A} (ii)), on $\|f^o\|_{\partial\Omega^o}$, on the constant $C_F \in ]0,+\infty[$ provided by the growth
	condition \eqref{conditionF1F2} (cf. \eqref{inequalityNF1NF2}), on the norm of the bounded linear operator $v^+_{\Omega^o}[\cdot]_{|\partial\Omega^i}$ from $C^0(\partial\Omega^o)$ to $C^0(\partial\Omega^i)$, and on the norm of the bounded linear operator $V_{\partial\Omega^i}$ from $C^0(\partial\Omega^i)$ into itself, such that
	\[
	\begin{split}
	&\|(\mu^o,\mu^i,\eta^i,\rho^o,\rho^i)\|_{C^0(\partial\Omega^o) \times (C^0(\partial\Omega^i))^2 \times \R^2}\\ & \qquad \leq 
	C_1 (C_2 + \|(\mu^o,\mu^i,\eta^i,\rho^o,\rho^i)\|_{C^0(\partial\Omega^o) \times (C^0(\partial\Omega^i))^2 \times \R^2} )^\delta\, .
	\end{split}
	\]
	Then, by a straightforward calculation, we can show the existence of a constant $C>0$ such that inequality \eqref{mu<C} holds true (cf. Lanza de Cristoforis \cite[proof of Thm.~7.2]{La07}).	Hence, by the  Leray-Schauder Theorem \ref{Thm Leray Schauder} there exists at least one solution \[(\mu^o_0,\mu^i_0,\eta^i_0,\rho^o_0,\rho^i_0) \in C^0(\partial\Omega^o)_0 \times C^0(\partial\Omega^i) \times C^0(\partial\Omega^i)_0 \times \R^2\] of  $(\mu^o,\mu^i,\eta^i,\rho^o,\rho^i) = T_A(\mu^o,\mu^i,\eta^i,\rho^o,\rho^i)$.	Finally, by the definition of $T_A$ (cf. \eqref{T}), we conclude that $(\mu^o_0,\mu^i_0,\eta^i_0,\rho^o_0,\rho^i_0)$ is a solution in $C^0(\partial\Omega^o)_0 \times C^0(\partial\Omega^i) \times C^0(\partial\Omega^i)_0 \times \R^2$ of the nonlinear system \eqref{princintsys}.
\end{proof}

In what follows we will exploit a continuity condition on the superposition operators generated by $F_1$ and $F_2$, namely
\begin{equation}\label{conditionNF}
\begin{split}
\bullet & \, \mbox{The superposition operators } \mathcal{N}_{F_1} \mbox{ and }  \mathcal{N}_{F_2} \mbox{ are continuous  from } \\
& (C^{0,\alpha}(\partial\Omega^i))^2 \mbox{ into } C^{0,\alpha}(\partial\Omega^i).
\end{split}
\end{equation}
For conditions on $F_1$ and $F_2$ which imply the validity of assumption \eqref{conditionNF},
we refer to Appell and Zabrejko \cite[Ch.~8]{ApZa90} and to  Valent \cite[Chap. II]{Va88}. Then we can prove a regularity result for the fixed point provided by Proposition \ref{prop mu_0}, and, thus, an existence result for  problem \eqref{princeq}.

\begin{prop}\label{prop u^o_0,u^i_0}
	Let A be as in \eqref{Acondition}. Let assumptions \eqref{conditionF1F2} and \eqref{conditionNF} hold. Then the nonlinear system \eqref{princintsys} has at least one solution $(\mu^o_0,\mu^i_0,\eta^i_0,\rho^o_0,\rho^i_0) \in C^{0,\alpha}(\partial\Omega^o)_0 \times C^{0,\alpha}(\partial\Omega^i) \times C^{0,\alpha}(\partial\Omega^i)_0 \times \R^2$. In particular, problem \eqref{princeq} has at least one solution $(u^o_0,u^i_0) \in C^{1,\alpha}(\overline{\Omega^o} \setminus \Omega^i) \times C^{1,\alpha}(\overline{\Omega^i})$ given by
	\begin{equation}\label{u^o_0,u^i_0}
	(u^o_0,u^i_0) \equiv  (U^o_{\Omega^i}[\mu^o_0,\mu^i_0,\eta^i_0,\rho^o_0,\rho^i_0],U^i_{\Omega^i}[\mu^o_0,\mu^i_0,\eta^i_0,\rho^o_0,\rho^i_0])
	\end{equation}
	where the pair $(U^o_{\Omega^i}[\cdot,\cdot,\cdot,\cdot,\cdot],U^i_{\Omega^i}[\cdot,\cdot,\cdot,\cdot,\cdot])$ is defined by \eqref{U^o,U^i}.
\end{prop}

\begin{proof} \ 
	Let $T_A$ be as in \eqref{T}. By Proposition \ref{prop mu_0},  we deduce the existence of a quintuple $(\mu^o_0,\mu^i_0,\eta^i_0,\rho^o_0,\rho^i_0) \in C^0(\partial\Omega^o)_0 \times C^0(\partial\Omega^i) \times C^0(\partial\Omega^i)_0 \times \R^2$ such that $
	(\mu^o_0,\mu^i_0,\eta^i_0,\rho^o_0,\rho^i_0) = T_A(\mu^o_0,\mu^i_0,\eta^i_0,\rho^o_0,\rho^i_0).
$ By the mapping properties of integral operators with real analytic kernel and no singularities (cf. Lanza de Cristoforis and Musolino \cite[Prop. 4.1]{LaMu13}), $v^+_{\Omega^o}[\mu^o_0]_{|\partial\Omega^i}$ belongs to $C^{0,\alpha}(\partial\Omega^i)$. By classical results in potential theory (cf.~Miranda \cite[Chap.~II, \S14, III]{Mi70}), $V_{\Omega^i}[\mu^i_0]$ and $V_{\Omega^i}[\eta^i_0]$ belong to $C^{0,\alpha}(\partial\Omega^i)$. Then, by condition \eqref{conditionNF} and by the membership of $A \in M_2(C^{0,\alpha} (\partial\Omega^i))$ and of $f^o \in C^{0,\alpha} (\partial\Omega^o)$, we obtain that
	\begin{equation*}
	\begin{split}
	&\begin{pmatrix}
	f^o
	\\
	\mathcal{N}_{F_1}(v^+_{\Omega^o}[\mu^o_0]_{|\partial\Omega^i} + V_{\partial\Omega^i}[\mu^i_0] +\rho^o_0 ,V_{\partial\Omega^i}[\eta^i_0] +\rho^i_0)
	\\
	\mathcal{N}_{F_2}(v^+_{\Omega^o}[\mu^o_0]_{|\partial\Omega^i} + V_{\partial\Omega^i}[\mu^i_0] +\rho^o_0 ,V_{\partial\Omega^i}[\eta^i_0] +\rho^i_0)
	\end{pmatrix}
	\\ &\qquad - \begin{pmatrix}
	0 & 0 & 0 \\
	0 & A_{11} & A_{12} \\
	0 & A_{21} & A_{22} 
	\end{pmatrix} 
	\begin{pmatrix}
	0
	\\
	v^+_{\Omega^o}[\mu^o_0]_{|\partial\Omega^i} + V_{\partial\Omega^i}[\mu^i_0] +\rho^o_0
	\\
	V_{\partial\Omega^i}[\eta^i_0] +\rho^i_0
	\end{pmatrix}
	\end{split}
	\end{equation*}
	belongs to the product space $C^{0,\alpha}(\partial\Omega^o) \times (C^{0,\alpha}(\partial\Omega^i))^2$. Finally, by the invertibility of the operator $J_A$ from $C^{0,\alpha}(\partial\Omega^o)_0 \times C^{0,\alpha}(\partial\Omega^i) \times C^{0,\alpha}(\partial\Omega^i)_0 \times \R^2$ to $C^{0,\alpha}(\partial\Omega^o) \times (C^{0,\alpha}(\partial\Omega^i))^2$, we obtain that $(\mu^o_0,\mu^i_0,\eta^i_0,\rho^o_0,\rho^i_0) \in C^{0,\alpha}(\partial\Omega^o)_0 \times C^{0,\alpha}(\partial\Omega^i) \times C^{0,\alpha}(\partial\Omega^i)_0 \times \R^2$.	In particular, by Proposition \ref{propintsys} we deduce that the pair given by \eqref{u^o_0,u^i_0} is a solution of   \eqref{princeq} (cf.~\eqref{T}).
\end{proof}

\section{The perturbed transmission problem (\ref{princeqpertu})}\label{sec princeqpertu}

This section is devoted to the study of the perturbed transmission problem \eqref{princeqpertu}. We introduce the map $M=(M_1,M_2,M_3)$ from $\mathcal{A}^{\Omega^o}_{\partial\Omega^i} \times C^{0,\alpha}(\partial\Omega^o)_0 \times C^{0,\alpha}(\partial\Omega^i) \times C^{0,\alpha}(\partial\Omega^i)_0 \times \R^2$ to $C^{0,\alpha}(\partial\Omega^o) \times (C^{0,\alpha}(\partial\Omega^i))^2$ defined by
\begin{equation}\label{M}
\begin{aligned}
& M_1[\phi,\mu^o,\mu^i,\eta^i,\rho^o,\rho^i](x) \\ & \equiv \left( -\frac{1}{2} I + W^\ast_{\partial\Omega^o} \right) [\mu^o] (x) + \nu_{\Omega^o}(x) \cdot \nabla  v^-_{\Omega^i[\phi]}[\mu^i \circ \phi^{(-1)}] (x) - f^o(x) \quad \forall x \in \partial\Omega^o
\\
& M_2[\phi,\mu^o,\mu^i,\eta^i,\rho^o,\rho^i] (t) \\&\equiv
\left( \frac{1}{2} I + W^\ast_{\partial\Omega^i[\phi]} \right) [\mu^i \circ \phi^{(-1)}] (\phi(t)) + \nu_{\Omega^i[\phi]}(\phi(t)) \cdot \nabla  v^+_{\Omega^o}[\mu^o](\phi(t))
\\
& \qquad \qquad
- F_1\bigg(t,v^+_{\Omega^o}[\mu^o](\phi(t)) + V_{\partial\Omega^i[\phi]}[\mu^i\circ \phi^{(-1)}](\phi(t)) +\rho^o , \\
& \qquad \qquad \qquad \qquad V_{\partial\Omega^i[\phi]}[\eta^i\circ \phi^{(-1)}](\phi(t)) +\rho^i \bigg)
 \quad \forall t \in \partial\Omega^i
\\
& M_3[\phi,\mu^o,\mu^i,\eta^i,\rho^o,\rho^i] (t) \\ &\equiv
\left( -\frac{1}{2} I + W^\ast_{\partial\Omega^i[\phi]} \right) [\eta^i \circ \phi^{(-1)}] (\phi(t))
\\
& \qquad \qquad
- F_2\bigg(t,v^+_{\Omega^o}[\mu^o](\phi(t)) + V_{\partial\Omega^i[\phi]}[\mu^i\circ \phi^{(-1)}](\phi(t)) +\rho^o ,\\
& \qquad \qquad \qquad \qquad  V_{\partial\Omega^i[\phi]}[\eta^i\circ \phi^{(-1)}](\phi(t)) +\rho^i \bigg ) 
 \quad\forall t \in \partial\Omega^i
\end{aligned}
\end{equation}
for all $(\phi,\mu^o,\mu^i,\eta^i,\rho^o,\rho^i) \in \mathcal{A}^{\Omega^o}_{\partial\Omega^i} \times C^{0,\alpha}(\partial\Omega^o)_0 \times C^{0,\alpha}(\partial\Omega^i) \times C^{0,\alpha}(\partial\Omega^i)_0 \times \R^2$. We incidentally observe that by the definition of $\Omega^i[\phi]$ we have that $\partial \Omega^i[\phi]=\phi(\partial\Omega^i)$. 
Then, by the definition of $M$, we can deduce the following result.

\begin{prop}\label{M=0prop}
	Let A be as in \eqref{Acondition}. Let assumptions \eqref{conditionF1F2} and \eqref{conditionNF} hold. Let
	\begin{equation*}
	(\phi,\mu^o,\mu^i,\eta^i,\rho^o,\rho^i) \in \mathcal{A}^{\Omega^o}_{\partial\Omega^i} \times C^{0,\alpha}(\partial\Omega^o)_0 \times C^{0,\alpha}(\partial\Omega^i) \times C^{0,\alpha}(\partial\Omega^i)_0 \times \R^2.
	\end{equation*}
	Then the pair of functions
	\begin{equation*}
	(U^o_{\Omega^i[\phi]}[\mu^o,\mu^i\circ \phi^{(-1)},\eta^i\circ \phi^{(-1)},\rho^o,\rho^i],U^i_{\Omega^i[\phi]}[\mu^o,\mu^i\circ \phi^{(-1)},\eta^i\circ \phi^{(-1)},\rho^o,\rho^i])
	\end{equation*}
	defined by \eqref{U^o,U^i} is a solution of problem \eqref{princeqpertu} if and only if 
	\begin{equation}\label{M=0}
	M[\phi,\mu^o,\mu^i,\eta^i,\rho^o,\rho^i] = (0,0,0).
	\end{equation}
	In particular, equation
	\begin{equation}\label{M_0=0}
	M[\phi_0,\mu^o,\mu^i,\eta^i,\rho^o,\rho^i] = (0,0,0)
	\end{equation}
	is equivalent to the system \eqref{princintsys} and has a solution $(\mu^o_0,\mu^i_0,\eta^i_0,\rho^o_0,\rho^i_0) \in  C^{0,\alpha}(\partial\Omega^o)_0 \times C^{0,\alpha}(\partial\Omega^i) \times C^{0,\alpha}(\partial\Omega^i)_0 \times \R^2$  (recall that $\phi_0\equiv \text{id}_{\partial\Omega^i}$).
\end{prop}

\begin{proof} \ 
	We first observe that, by the regularity of $\phi \in \mathcal{A}^{\Omega^o}_{\partial\Omega^i}$, if $(\mu^o,\mu^i,\eta^i,\rho^o,\rho^i) \in  C^{0,\alpha}(\partial\Omega^o)_0 \times C^{0,\alpha}(\partial\Omega^i) \times C^{0,\alpha}(\partial\Omega^i)_0 \times \R^2$, then
	\begin{equation*}
	(\mu^o,\mu^i,\eta^i,\rho^o,\rho^i) \in C^{0,\alpha}(\partial\Omega^o)_0 \times C^{0,\alpha}(\partial\Omega^i) \times C^{0,\alpha}(\partial\Omega^i)_0 \times \R^2.
	\end{equation*}
	Moreover, since $\overline{\Omega^i[\phi]} \subset \Omega^o$, we can apply Lemma \ref{rapprharm} with $\Omega=\Omega^i[\phi]$. Then by the jump relations for the single layer potential (cf.~Theorem \ref{sdp} (iii)), by a change of variable on $\phi(\partial \Omega^i)$ and by the definition of $M$ (cf.~\eqref{M}), we obtain that the pair of functions
	\begin{align*}
	& U^o_{\Omega^i[\phi]}[\mu^o,\mu^i\circ \phi^{(-1)},\eta^i\circ \phi^{(-1)},\rho^o,\rho^i] = (v^+_{\Omega^o} [\mu^o] + v^-_{\Omega^i[\phi]}[\mu^i\circ \phi^{(-1)}] + \rho^o)_{| \overline{\Omega^o} \setminus \Omega^i[\phi]},
	\\
	& U^i_{\Omega^i[\phi]}[\mu^o,\mu^i\circ \phi^{(-1)},\eta^i\circ \phi^{(-1)},\rho^o,\rho^i] = v^+_{\Omega^i[\phi]}[\eta^i\circ \phi^{(-1)}] + \rho^i
	\end{align*}
	is a solution of problem \eqref{princeqpertu} if and only if \eqref{M=0} is satisfied.  Finally,  since $\phi_0\equiv \text{id}_{\partial\Omega^i}$ (cf. \eqref{phi0}) and by the definition of $J_A$ (cf.~\eqref{J_A eq}), we obtain that, for all $(\mu^o,\mu^i,\eta^i,\rho^o,\rho^i) \in  C^{0,\alpha}(\partial\Omega^o)_0 \times C^{0,\alpha}(\partial\Omega^i) \times C^{0,\alpha}(\partial\Omega^i)_0 \times \R^2$,
	equation \eqref{M_0=0} is equivalent to the system \eqref{princintsys}. Then the existence of a solution $(\mu^o_0,\mu^i_0,\eta^i_0,\rho^o_0,\rho^i_0)$ of \eqref{M_0=0} follows by Proposition \ref{prop mu_0}.
\end{proof}

By Proposition \ref{M=0prop}, the study of problem \eqref{princeqpertu} is reduced to that of equation \eqref{M=0}. 
We now wish to apply the Implicit Function Theorem for real analytic maps in Banach spaces (cf. Deimling \cite[Thm. 15.3]{De85}) to equation \eqref{M=0} around the value $\phi_0$. As a first step we have to analyse the regularity of the map $M$.

In what follows we will assume the following:
\begin{equation}\label{conditionNF*}
\begin{split}
\bullet & \,\mbox{The superposition operators } \mathcal{N}_{F_1} \mbox{ and }  \mathcal{N}_{F_2} \mbox{ are real analytic from } \\
& (C^{0,\alpha}(\partial\Omega^i))^2 \mbox{ into } C^{0,\alpha}(\partial\Omega^i).
\end{split}
\end{equation}

For conditions on $F_1$ and $F_2$ which imply the validity of assumption \eqref{conditionNF*},
we refer to Valent \cite[Chap. II]{Va88}. We now show that $M$ is real analytic.

\begin{prop}\label{Mrealanal}
	Let assumption \eqref{conditionNF*} holds. Then the map $M$ is real analytic from $\mathcal{A}^{\Omega^o}_{\partial\Omega^i} \times C^{0,\alpha}(\partial\Omega^o)_0 \times C^{0,\alpha}(\partial\Omega^i) \times C^{0,\alpha}(\partial\Omega^i)_0 \times \R^2$ to $C^{0,\alpha}(\partial\Omega^o) \times (C^{0,\alpha}(\partial\Omega^i))^2$.
\end{prop}

\begin{proof} \ 
	 We only prove the analyticity of $M_2$. The analyticity of $M_1$ and of $M_3$ can be proved similarly and it is left to the reader.
	 Therefore, we now analyse $M_2$. The map from $\mathcal{A}^{\Omega^o}_{\partial\Omega^i} \times C^{0,\alpha}(\partial\Omega^i)$ to $C^{0,\alpha}(\partial\Omega^i)$  that  takes $(\phi,\mu^i)$ to the function of the variable $t\in\partial\Omega^i$ defined by
    \[
	\begin{split}
	&\left( \frac{1}{2} I + W^\ast_{\partial\Omega^i[\phi]} \right) [\mu^i \circ \phi^{(-1)}] (\phi(t)) = \frac{1}{2}\mu^i(t)  + W^\ast_{\partial\Omega^i[\phi]} [ \mu^i\circ \phi^{(-1)}] (\phi(t)) 
	\\
	&= \frac{1}{2}\mu^i(t) + \int_{\partial\Omega^i} (\nu_{\Omega^i[\phi]}(\phi(t))) \cdot \nabla S_n(\phi(t)-\phi(s))) \,\mu^i(s)\, \tilde{\sigma}_n[\phi](s) \,d\sigma_s
	\end{split}
    \]
	is real analytic by the real analyticity result for the dependence of layer potentials upon perturbation of the support and of the density of Lanza de Cristoforis and Rossi \cite[Thm.~3.12]{LaRo04} and Lanza de Cristoforis \cite[Prop.~7]{La07-2} (see also Lemma \ref{lemmanotation}). 
	The map from $\mathcal{A}^{\Omega^o}_{\partial\Omega^i} \times C^{0,\alpha}(\partial\Omega^o)$ to $C^{0,\alpha}(\partial\Omega^i)$  that  takes $(\phi,\mu^o)$ to the function of the variable $t \in\partial\Omega^i$ defined by
    \[
	\nu_{\Omega^i[\phi]}(\phi(t)) \cdot \nabla  v^+_{\Omega^o}[\mu^o](\phi(t))
	= \int_{\partial\Omega^o} (\nu_{\Omega^i[\phi]}(\phi(t)) \cdot \nabla  S_n(\phi(t)-y))  \, \mu^o(y) \,d\sigma_y 
    \]
	can be proven to be
	real analytic  by 
	 the properties of integral operators with real analytic kernels and no singularities (see
	Lanza de Cristoforis and Musolino \cite[Prop. 4.1]{LaMu13}).
	For the third term of $M_2$ we proceed in this way. The map from $\mathcal{A}^{\Omega^o}_{\partial\Omega^i} \times C^{0,\alpha}(\partial\Omega^o)$ to $C^{0,\alpha}(\partial\Omega^i)$  that  takes $(\phi,\mu^o)$ to the function of the variable $t\in\partial\Omega^i$ defined by
    \[
	v^+_{\Omega^o}[\mu^o](\phi(t)) =  \int_{\partial\Omega^o}  S_n(\phi(t)-y)  \, \mu^o(y) \,d\sigma_y 
    \]
	can be proven to be real analytic  by 
	  the properties of integral operators with real analytic kernels and no singularities (see
	Lanza de Cristoforis and Musolino \cite[Prop. 4.1]{LaMu13}). The map from $\mathcal{A}^{\Omega^o}_{\partial\Omega^i} \times C^{0,\alpha}(\partial\Omega^i)$ to $C^{0,\alpha}(\partial\Omega^i)$  that  takes $(\phi,\mu^o)$ to the function of the variable $t\in\partial\Omega^i$ defined by
    \[
	V_{\partial\Omega^i[\phi]}[\mu^i\circ \phi^{(-1)}](\phi(t)) = \int_{\phi(\partial\Omega^i)}  S_n(\phi(t)-y)  \, \mu^i\circ \phi^{(-1)}(y) \,d\sigma_y 
    \]
	is real analytic by a result of real analytic dependence for the single layer potential upon perturbation of the support and of the density (see Lanza de Cristoforis and Rossi \cite[Thm.~3.12]{LaRo04}, Lanza de Cristoforis \cite[Prop.~7]{La07-2}). Similarly we can  treat $V_{\partial\Omega^i[\phi]}[\eta^i\circ \phi^{(-1)}](\phi(\cdot))$. Hence, by the real analyticity of the composition of real analytic maps and  by   \eqref{conditionNF*}, we conclude that the map from $\mathcal{A}^{\Omega^o}_{\partial\Omega^i} \times C^{0,\alpha}(\partial\Omega^o)_0 \times C^{0,\alpha}(\partial\Omega^i) \times C^{0,\alpha}(\partial\Omega^i)_0 \times \R^2$ to $C^{0,\alpha}(\partial\Omega^i)$  that  takes a sextuple $(\phi,\mu^o,\mu^i,\eta^i,\rho^o,\rho^i)$ to the function
	\[
	\begin{split}
	\mathcal{N}_{F_1}\bigg( v^+_{\Omega^o}[\mu^o](\phi(\cdot))_{|\partial\Omega^i} + V_{\partial\Omega^i[\phi]}[\mu^i\circ \phi^{(-1)}]&(\phi(\cdot)) +\rho^o ,\\
	& V_{\partial\Omega^i[\phi]}[\eta^i\circ \phi^{(-1)}](\phi(\cdot)) +\rho^i   \bigg)
	\end{split}
	\]
	is real analytic. As a consequence $M_2$ is real analytic. 
\end{proof}

It will be convenient to consider $F_1$, $F_2$ as two components of a vector field on $\partial\Omega^i\times \R^2$. We denote by  $F$ the function from  $\partial\Omega^i\times \R^2$ to $\R^2$ defined by
\begin{equation*}\label{defF}
F(t,\zeta_1,\zeta_2) = (F_1(t,\zeta_1,\zeta_2),F_2(t,\zeta_1,\zeta_2)) \quad \forall (t,\zeta_1,\zeta_2) \in \partial\Omega^i\times\R^2\, .
\end{equation*} 
Clearly, we can extend the definition of the superposition operator (cf. Section \ref{notation}) in a natural way, {i.e.}, by setting
\begin{equation*}\label{defNF}
\mathcal{N}_F : (C^{0,\alpha}(\partial\Omega^i))^2 \to (C^{0,\alpha}(\partial\Omega^i))^2, \, \mathcal{N}_F \equiv (\mathcal{N}_{F_1},\mathcal{N}_{F_2})\, .
\end{equation*}
Now let $(\mu^o_0,\mu^i_0,\eta^i_0,\rho^o_0,\rho^i_0) \in C^{0,\alpha}(\partial\Omega^o)_0 \times C^{0,\alpha}(\partial\Omega^i) \times C^{0,\alpha}(\partial\Omega^i)_0 \times \R^2$ be as in Proposition \ref{prop u^o_0,u^i_0}. By standard calculus in Banach space, we have the following formula regarding the first order differential of $\mathcal{N}_F$: 
\[
d\mathcal{N}_F(v^+_{\Omega^o}[\mu^o_0]_{|\partial\Omega^i} +  V_{\Omega^i}[\mu^i_0] + \rho^o_0 , V_{\Omega^i}[\eta^i_0] +\rho^i_0) .(h_1,h_2) = A_{\mathcal{N}_F,0}  \begin{pmatrix}
h_1
\\
h_2
\end{pmatrix}
\]
for all $(h_1,h_2) \in (C^{0,\alpha}(\partial\Omega^i))^2$, where 
\begin{equation}\label{A_{N_F}}
A_{\mathcal{N}_F,0} \equiv \begin{pmatrix}
\mathcal{N}_{\partial_{\zeta_1}F_1}(\alpha^1_0 , \alpha^2_0)
&
\mathcal{N}_{\partial_{\zeta_2}F_1}(\alpha^1_0 , \alpha^2_0)
\\
\mathcal{N}_{\partial_{\zeta_1}F_2}(\alpha^1_0 , \alpha^2_0)
&
\mathcal{N}_{\partial_{\zeta_2}F_2}(\alpha^1_0 , \alpha^2_0)
\end{pmatrix}
\end{equation}
and $\alpha^1_0$ and $\alpha^2_0$ are the functions from $\partial\Omega^i$ to $\R$ defined by
\begin{equation}\label{A_{N_F}bis}
\begin{split}
\alpha^1_0 \equiv v^+_{\Omega^o}[\mu^o_0]_{|\partial\Omega^i} +  V_{\Omega^i}[\mu^i_0] + \rho^o_0, \quad
\alpha^2_0 \equiv V_{\Omega^i}[\eta^i_0] +\rho^i_0.
\end{split}
\end{equation}
We will require that the matrix $A_{\mathcal{N}_F,0}$ given by \eqref{A_{N_F}}-\eqref{A_{N_F}bis} satisfies assumption \eqref{Acondition}.
In particular, we notice that assumption \eqref{conditionNF*} implies the validity of the first of the three conditions of \eqref{Acondition} for the matrix $A_{\mathcal{N}_F,0}$.  In order to apply the Implicit Function Theorem (cf. Deimling \cite[Thm. 15.3]{De85}) to equation \eqref{M=0} we need to prove the invertibility of the partial differential of  $M$. 

\begin{prop}\label{diffMprop}
	Let assumptions \eqref{conditionF1F2} and \eqref{conditionNF*} hold. Let $(\mu^o_0,\mu^i_0,\eta^i_0,\rho^o_0,\rho^i_0) \in C^{0,\alpha}(\partial\Omega^o)_0 \times C^{0,\alpha}(\partial\Omega^i) \times C^{0,\alpha}(\partial\Omega^i)_0 \times \R^2$  be as in Proposition \ref{prop u^o_0,u^i_0}. Let $A_{\mathcal{N}_F,0}$ be as in \eqref{A_{N_F}}-\eqref{A_{N_F}bis} and assume that satisfies assumption \eqref{Acondition}. Then the partial differential of $M$ with respect to $(\mu^o,\mu^i,\eta^i,\rho^o,\rho^i)$ evaluated at the point $(\phi_0,\mu^o_0,\mu^i_0,\eta^i_0,\rho^o_0,\rho^i_0)$, which we denote by 
	\begin{equation}\label{partdiff M}
	\partial_{(\mu^o, \mu^i, \eta^i, \rho^o,\rho^i)} M[\phi_0,\mu^o_0, \mu^i_0 ,\eta^i_0,\rho^o_0,\rho^i_0],
	\end{equation}
	is an isomorphism from $C^{0,\alpha}(\partial\Omega^o)_0 \times C^{0,\alpha}(\partial\Omega^i) \times C^{0,\alpha}(\partial\Omega^i)_0 \times \R^2$ to $C^{0,\alpha}(\partial\Omega^o) \times (C^{0,\alpha}(\partial\Omega^i))^2$.
\end{prop}

\begin{proof} \ 
	By standard calculus in Banach spaces, we can verify that the partial differential \eqref{partdiff M} is the linear and continuous operator defined by
	\begin{align*}
	& \partial_{(\mu^o,\mu^i,\eta^i,\rho^o,\rho^i)} M_1[\phi_0,\mu^o_0,\mu^i_0,\eta^i_0,\rho^o_0,\rho^i_0]. (\tilde{\mu}^o,\tilde{\mu}^i,\tilde{\eta}^i,\tilde{\rho}^o,\tilde{\rho}^i)(x) 
	\\
	&\qquad
	= \left( -\frac{1}{2} I + W^\ast_{\Omega^o} \right) [\tilde{\mu}^o] (x) 
	+ \nu_{\Omega^o}(x) \cdot \nabla  v^-_{\Omega^i}[\tilde{\mu}^i] (x) \qquad \forall x \in \partial\Omega^o
	\\
	& \partial_{(\mu^o,\mu^i,\eta^i,\rho^o,\rho^i)} M_2[\phi_0,\mu^o_0,\mu^i_0,\eta^i_0,\rho^o_0,\rho^i_0]. (\tilde{\mu}^o,\tilde{\mu}^i,\tilde{\eta}^i,\tilde{\rho}^o,\tilde{\rho}^i)(t)
	\\
	&\qquad
	= \left( \frac{1}{2} I + W^\ast_{\Omega^i} \right) [\tilde{\mu}^i] (t) 
	+ \nu_{\Omega^i}(t) \cdot \nabla  v^+_{\Omega^o}[\mu^o](t)	
	\\
	& \qquad -\partial_{\zeta_1}F_1\left(t,v^+_{\Omega^o}[\mu^o_0](t) + V_{\Omega^i}[\mu^i_0](t) + \rho^o_0 , V_{\Omega^i}[\eta^i_0](t) +\rho^i_0 \right) \, \\
	& \qquad \qquad \times \left(v^+_{\Omega^o}[\tilde{\mu}^o](t) + V_{\Omega^i}[\tilde{\mu}^i](t) + \tilde{\rho}^o\right)
	\\
	& \qquad
	-\partial_{\zeta_2}F_1\left(t,v^+_{\Omega^o}[\mu^o_0](t) + V_{\Omega^i}[\mu^i_0](t) + \rho^o_0 , V_{\Omega^i}[\eta^i_0](t) +\rho^i_0 \right) \, \\
	& \qquad \qquad \times \left( V_{\Omega^i}[\tilde{\eta}^i](t) + \tilde{\rho}^i\right)
	\qquad \forall t \in \partial\Omega^i
	\\
	& \partial_{(\mu^o,\mu^i,\eta^i,\rho^o,\rho^i)} M_3[\phi_0,\mu^o_0,\mu^i_0,\eta^i_0,\rho^o_0,\rho^i_0]. (\tilde{\mu}^o,\tilde{\mu}^i,\tilde{\eta}^i,\tilde{\rho}^o,\tilde{\rho}^i)(t) 
	\\
	&\qquad
	= \left( -\frac{1}{2} I + W^\ast_{\Omega^i} \right) [\tilde{\eta}^i] (t) 
	\\
	& \qquad -\partial_{\zeta_1}F_2\left(t,v^+_{\Omega^o}[\mu^o_0](t) + V_{\Omega^i}[\mu^i_0](t) + \rho^o_0 , V_{\Omega^i}[\eta^i_0](t) +\rho^i_0 \right) \, \\
	& \qquad \qquad \times  \left(v^+_{\Omega^o}[\tilde{\mu}^o](t) + V_{\Omega^i}[\tilde{\mu}^i](t) + \tilde{\rho}^o\right)
	\\
	& \qquad
	-\partial_{\zeta_2}F_2\left(t,v^+_{\Omega^o}[\mu^o_0](t) + V_{\Omega^i}[\mu^i_0](t) + \rho^o_0 , V_{\Omega^i}[\eta^i_0](t) +\rho^i_0 \right) \, \\
	 & \qquad \qquad \times \left( V_{\Omega^i}[\tilde{\eta}^i](t) + \tilde{\rho}^i\right)
	\qquad \forall t \in \partial\Omega^i
	\end{align*}
	for all $(\tilde{\mu}^o,\tilde{\mu}^i,\tilde{\eta}^i,\tilde{\rho}^o,\tilde{\rho}^i) \in C^{0,\alpha}(\partial\Omega^o)_0 \times C^{0,\alpha}(\partial\Omega^i) \times C^{0,\alpha}(\partial\Omega^i)_0 \times \R^2$. Then, by Proposition \ref{J_A} with $A=A_{\mathcal{N}_F,0}$ (cf.~\eqref{J_A eq}) and   since $A_{\mathcal{N}_F,0}$ satisfies  \eqref{Acondition}, we conclude that 
	 $
	\partial_{(\mu^o,\mu^i,\eta^i,\rho^o,\rho^i)} M[\phi_0,\mu^o_0,\mu^i_0,\eta^i_0,\rho^o_0,\rho^i_0]$ 
	is an isomorphism of Banach spaces.
\end{proof}

 By Propositions  \ref{M=0prop}, \ref{Mrealanal},   \ref{diffMprop}, and by applying the Implicit Function Theorem for real analytic functions in Banach spaces (cf.~Deimling \cite[Thm.~15.3]{De85}) to  equation \eqref{M=0}, we deduce the following real analyticity result for the dependence of the densities in the integral representation fomula for the solutions of problem \eqref{princeqpertu} upon the perturbation of the shape of the inclusion $\Omega^i$.

\begin{teo}\label{M^oteo}
	Let assumptions \eqref{conditionF1F2} and \eqref{conditionNF*} hold. Let
	 $(\mu^o_0,\mu^i_0,\eta^i_0,\rho^o_0,\rho^i_0) \in C^{0,\alpha}(\partial\Omega^o)_0 \times C^{0,\alpha}(\partial\Omega^i) \times C^{0,\alpha}(\partial\Omega^i)_0 \times \R^2$ be as in Proposition \ref{prop u^o_0,u^i_0}. Let $A_{\mathcal{N}_F,0}$ be as in \eqref{A_{N_F}}-\eqref{A_{N_F}bis} and assume that satisfies assumption \eqref{Acondition}. Then, there exist two open neighbourhoods $Q_0$ of $\phi_0$ in $\mathcal{A}^{\Omega^o}_{\partial\Omega^i}$ and $U_0$ of $(\mu^o_0,\mu^i_0,\eta^i_0,\rho^o_0,\rho^i_0)$ in $C^{0,\alpha}(\partial\Omega^o)_0 \times C^{0,\alpha}(\partial\Omega^i) \times C^{0,\alpha}(\partial\Omega^i)_0 \times \R^2$, and a real analytic map  $
	\Lambda \equiv (M^o,M^i,N^i,R^o,R^i): Q_0 \to U_0$ 	such that the set of zeros of $M$ in $Q_0 \times U_0$ coincides with the graph of the function $\Lambda$.  In particular,   
	\[
	\Lambda[\phi_0]=(M^o[\phi_0],M^i[\phi_0],N^i[\phi_0],R^o[\phi_0],R^i[\phi_0])= (\mu^o_0,\mu^i_0,\eta^i_0,\rho^o_0,\rho^i_0).
	\]
\end{teo}


We are now ready to exhibit a family of solutions of problem \eqref{princeqpertu}. 

\begin{defin}\label{u^o_phi,u^i_phi def}
	Let assumptions \eqref{conditionF1F2} and \eqref{conditionNF*} hold. Let $A_{\mathcal{N}_F,0}$ be as in \eqref{A_{N_F}}-\eqref{A_{N_F}bis} and assume that satisfies assumption \eqref{Acondition}. Let $Q_0$ and $\Lambda\equiv (M^o,M^i,N^i,R^o,R^i)$ be as in Theorem \ref{M^oteo}. Then, for each $\phi \in Q_0$ we set
	\begin{align*}
	 u^o_\phi(x)  &= U^o_{\Omega^i[\phi]}[M^o[\phi],M^i[\phi]\circ \phi^{(-1)},N^i[\phi]\circ \phi^{(-1)},R^o[\phi],R^i[\phi]](x)  \\& \quad\qquad \qquad \qquad \qquad \qquad \qquad \qquad \qquad \qquad \qquad \qquad\qquad  \forall x \in \overline{\Omega^o} \setminus  \Omega^i[\phi],
	\\
	u^i_\phi(x)  &= U^i_{\Omega^i[\phi]}[M^o[\phi],M^i[\phi]\circ \phi^{(-1)},N^i[\phi]\circ \phi^{(-1)},R^o[\phi],R^i[\phi]](x) \quad  \forall x \in \overline{\Omega^i[\phi]},
	\end{align*}
	where the pair $(U^o_{\Omega^i[\phi]}[\cdot,\cdot,\cdot,\cdot,\cdot],U^i_{\Omega^i[\phi]}[\cdot,\cdot,\cdot,\cdot,\cdot])$ is defined by \eqref{U^o,U^i}.
\end{defin}
By Propositions \ref{prop u^o_0,u^i_0}, \ref{M=0prop}, and Theorem \ref{M^oteo},  we deduce the  following.
\begin{teo}\label{upertuex}
	Let assumptions \eqref{conditionF1F2} and \eqref{conditionNF*} hold. Let $A_{\mathcal{N}_F,0}$ be as in \eqref{A_{N_F}}-\eqref{A_{N_F}bis} and assume that satisfies assumption \eqref{Acondition}. Let $Q_0$ be as in Theorem \ref{M^oteo} and let $(u^o_\phi,u^i_\phi)$ be as in Definition \ref{u^o_phi,u^i_phi def}. Then, for all $\phi \in Q_0$, $(u^o_\phi,u^i_\phi) \in C^{1,\alpha}(\overline{\Omega^o} \setminus \Omega^i[\phi]) \times C^{1,\alpha}(\overline{\Omega^i[\phi]})$
	is a solution of problem \eqref{princeqpertu}. In particular $(u^o_{\phi_0},u^i_{\phi_0})= (u^o_0,u^i_0)$	is a solution of problem \eqref{princeq}.
\end{teo}

We are now ready to prove our main result, where we show that suitable restrictions of the functions  $u^o_\phi$ and $u^i_\phi$ depend real analytically on the parameter $\phi$ which determines the domain perturbation. 

\begin{teo}\label{upertuana}
	Let assumptions \eqref{conditionF1F2} and \eqref{conditionNF*} hold. Let $A_{\mathcal{N}_F,0}$ be as in \eqref{A_{N_F}}-\eqref{A_{N_F}bis} and assume that satisfies assumption \eqref{Acondition}. Let $Q_0$ be as in Theorem \ref{M^oteo} and let $(u^o_\phi,u^i_\phi)$ be as in Definition \ref{u^o_phi,u^i_phi def}. Then, the following statements hold.
	\begin{enumerate}
		\item[(i)] Let $\Omega_\mathtt{int}$ be a bounded open subset of $\Omega^o$. Let $Q_\mathtt{int} \subseteq Q_0$ be an open neighbourhood of $\phi_0$ such that  \[
		\overline{\Omega_\mathtt{int}} \subset {\Omega^i[\phi]} \quad \forall \phi \in Q_\mathtt{int}.
		\]
		Then the map from $Q_\mathtt{int}$ to $C^{1,\alpha}(\overline{\Omega_\mathtt{int}})$  that  takes $\phi$ to $u^i_{\phi| \overline{\Omega_\mathtt{int}}}$ is real analytic.
		
		\item[(ii)] Let $\Omega_\mathtt{ext}$ be a bounded open subset of $\Omega^o$. Let $Q_\mathtt{ext} \subseteq Q_0$ be an open neighbourhood of $\phi_0$ such that  \[
				\overline{\Omega_\mathtt{ext}} \subset \Omega^o \setminus \overline{\Omega^i[\phi]} \quad \forall \phi \in Q_\mathtt{ext}.\]	Then the map from $Q_\mathtt{ext}$ to $C^{1,\alpha}(\overline{\Omega_\mathtt{ext}})$  that  takes $\phi$ to $u^o_{\phi| \overline{\Omega_\mathtt{ext}}}$ is real analytic.
	\end{enumerate}
\end{teo}

\begin{proof} \ 
 We  prove (i).  By Definition \ref{u^o_phi,u^i_phi def}, by \eqref{U^o,U^i} and by Lemma \ref{lemmanotation}, we have
	\begin{equation*}
	\begin{split}
	u^i_\phi(x) &= U^i_{\Omega^i[\phi]}[M^o[\phi],M^i[\phi]\circ \phi^{(-1)},N^i[\phi]\circ \phi^{(-1)},R^o[\phi],R^i[\phi]](x) 
	\\
	& = \int_{\partial\Omega^i}  S_n(x-\phi(s))  \, N^i[\phi](s) \, \tilde{\sigma}_n[\phi](s) \,d\sigma_s	+ R^i[\phi]	\qquad \forall x \in \overline{\Omega^i[\phi]}
	\end{split}
	\end{equation*}
	and for all $\phi \in Q_0$. By the assumption $Q_\mathtt{int} \subseteq Q_0$ and Theorem \ref{M^oteo}, we know that the map from $Q_\mathtt{int}$ to $\R$  that  takes $\phi$ to $R^i[\phi]$ is real analytic. Moreover, by the real analyticity of $N^i[\cdot]$ (cf.~Theorem \ref{M^oteo}) and by the properties of integral operators with real analytic kernels and no singularities (see Lanza de Cristoforis and Musolino \cite[Prop. 4.1]{LaMu13}), we can prove that the map from $Q_\mathtt{int}$ to $C^{1,\alpha}(\overline{\Omega_\mathtt{int}})$  that  takes $\phi$ to the function $
	\int_{\partial\Omega^i}  S_n(x-\phi(s))  \, N^i[\phi](s) \, \tilde{\sigma}_n[\phi](s) \,d\sigma_s$ of the variable $x \in \overline{\Omega_\mathtt{int}}$
	is real analytic (see also Lemma \ref{lemmanotation}). Hence, we deduce the validity of (i).  The proof of (ii) is similar and it is left to the reader. 
\end{proof}

\section*{Acknowledgements}
The authors are members of the ``Gruppo Nazionale per l'Analisi Matematica, la Probabilit\`a e le loro Applicazioni'' (GNAMPA) of the ``Istituto Nazionale di Alta Matematica'' (INdAM). R.M. acknowledges the support of the Project "Variational methods for stationary and evolution problems with singularities
and interfaces" (PRIN 2017) funded by the Italian Ministry of Education, University, and Research. P.M. acknowledges the support of the Project BIRD191739/19 ``Sensitivity analysis of partial differential equations in
the mathematical theory of electromagnetism'' (University of Padova), of the  ``INdAM GNAMPA Project 2020 - Analisi e ottimizzazione asintotica per autovalori in domini con piccoli buchi'', and of  the grant ``Challenges in Asymptotic and Shape Analysis - CASA'' (Ca' Foscari University of Venice).


\begin{thebibliography}{12}
	
	
	
\bibitem{ApZa90} J.~Appell, P.P.~Zabrejko, {\em Nonlinear superposition operators.} Cambridge Tracts in Mathematics, 95. Cambridge University Press, Cambridge (1990).	
	\bibitem{BaGa96}
	G.R.~Barrenechea, G.N.~Gatica, {\em On the coupling of boundary integral and finite element methods with nonlinear transmission conditions}, Appl.~Anal. {\bf 62} (1996), 181--210.
	
	\bibitem{BeWaWe90}
	H.~Berger, G.~Warnecke, W.L.~Wendland, {\em Finite elements for transonic potential flows},  Numer.~Methods Partial Differ.~Equations {\bf 6} (1990), 17--42.
	
%
%
	
	\bibitem{Co88}
	M. Costabel, {\em Boundary integral operators in Lipschitz domains: elementary results}, SIAM J. Math. Anal. {\bf 19} (1988), 613--626.
	
%


	\bibitem{CoSt90}
	M.~Costabel, E.P.~Stephan, {\em Coupling of finite and boundary element methods for an elastoplastic interface problem}, SIAM J. Numer. Anal. {\bf 27} (1990), 1212--1226.
	
\bibitem{DaLaMu21}
M.~Dalla Riva, M.~Lanza de Cristoforis, P. Musolino, {\it Singularly Perturbed Boundary Value Problems: A Functional Analytic Approach.} Springer Nature, Cham (2021).
	
	
	
	\bibitem{DaMi15}
	M.~Dalla Riva, G.~Mishuris, {\em Existence results for a nonlinear transmission problem}, J. Math. Anal. Appl. {\bf 430(2)} (2015), 718-741.
	
%
%
	
	\bibitem{DaMoMu19}
	M.~Dalla Riva, R.~Molinarolo, P.~Musolino, {\em Local uniqueness of the solutions for a singularly perturbed nonlinear nonautonomous transmission problem}, Nonlinear Anal. {\bf 191} (2020), 111645.
	
	\bibitem{De85} 
	K.~Deimling, {\em Nonlinear Functional Analysis}, Springer-Verlag, Berlin, (1985).
	
	
	\bibitem{Fo95} 
	G.B.~Folland, {\em Introduction to partial differential equations}, Second Edition, Princeton University Press, Princeton N.J. (1995).
	
	\bibitem{GaHs95}
	G.N.~Gatica, G.C.~Hsiao, {\em The uncoupling of boundary integral and finite element methods for nonlinear boundary value problems}, J.~Math.~Anal.~Appl. {\bf 189} (1995), 442--461.
	
	\bibitem{GiTr83} 
	D.~Gilbarg, N.S.~Trudinger, {\em Elliptic partial differential equations of second order}, Springer-Varlag, Berlin (1983).
	
	
	  \bibitem{HePi05}
A.~Henrot, M.~Pierre, {\em Variation et optimisation de formes}, Vol.~48 of
  Math\'ematiques \& Applications (Berlin) [Mathematics \& Applications],
  Springer, Berlin (2005). 
	
	\bibitem{He05}
	D. Henry, {\em Perturbation of the Boundary in Boundary-Value Problems of Partial Differential Equations}, London Mathematical Society Lecture Notes, Cambridge University Press  {\bf 318} (2005).
	
	\bibitem{Ke66}
	M.V. Keldysh, {\em On the solvability and stability of the Dirichlet problem}, Am. Math.
	Soc. Transl., {\bf II 51} (1966) 1--73; translation from Uspekhi Matematicheskikh Nauk, {\bf 8} (1941) 171--231.
	

	
	\bibitem{La02}
	M. Lanza~de Cristoforis,
	{\em Asymptotic behaviour of the conformal representation of a {J}ordan domain with a small hole in {S}chauder spaces}, Comput. Methods Funct. Theory {\bf 2} (2002), 1--27.
	
	
	\bibitem{La07}
	M.~Lanza~de~Cristoforis,  {\em Asymptotic behavior of the solutions of a nonlinear Robin problem
		for the Laplace operator in a domain with a small hole: a functional 
		analytic approach}, Complex Var.~Elliptic Equ., {\bf 52} (2007), 945--977. 
	
	\bibitem{La07-2}
	M. Lanza de Cristoforis, {\em Perturbation problems in potential theory, a functional analytic approach}, J. Appl. Funct. Anal, {\bf 2(3)} (2007), 197-222.
	
	
	\bibitem{La10}
	M.~Lanza de Cristoforis, {\em Asymptotic behaviour of the solutions of a non-linear transmission problem for the Laplace operator in a domain with a small hole. A functional analytic approach}, Complex Var. Elliptic Equ., {\bf 55(1)} (2010), 269-303.
	
	\bibitem{LaMu13}
	M. Lanza de Cristoforis, P. Musolino, {\em A real analyticity result for a nonlinear integral operator}, J. Int. Equ. Appl. {\bf 25(1)} (2013), 21-46.
	
	\bibitem{LaRo04}
	M. Lanza de Cristoforis, L. Rossi, {\em Real analytic dependence
	of simple and double layer potentials upon perturbation of the
	support and of the density}, J. Int. Equ. Appl., {\bf 16(2)}	(2004), 137–174.
	
	\bibitem{LuMu20} P. Luzzini, P. Musolino,  {\em Perturbation analysis of the effective conductivity of a periodic composite,} Netw. Heterog. Media {\bf 15(4)} (2020),  581--603.
	
	\bibitem{Mi70}
	C. Miranda, {\em Partial Differential Equations of Elliptic Type}, Springer-Verlag, Berlin (1970).
	
	\bibitem{MiRo00}
	V.V.~Mityushev, S.V.~Rogosin, {\em Constructive Methods for Linear and Nonlinear Boundary Value Problems for Analytic Functions},
	Chapman \& Hall/CRC Monogr. Surv. Pure Appl. Math. {\bf 108} (2000).
	
	\bibitem{Mo19}
	R. Molinarolo, {\em Existence of solutions for a singularly perturbed nonlinear non-autonomous transmission problem}, Electron. J. Diff. Equ. {\bf 2019(53)} (2019), 1--29.
	
	
	\bibitem{Ne83}
	J. Ne\v{c}as, {\em Introduction to the Theory of Nonlinear Elliptic Equations}, Teubner-Texte Math., {\bf 52} (1983).
	
	
	\bibitem{NoSo13}
A.~A. Novotny, J.~Soko{\l}{o}wski, {\em Topological derivatives in shape
  optimization}, Interaction of Mechanics and Mathematics, Springer, Heidelberg,
  (2013).
	
	
	\bibitem{Ro13}
	T. Roub\'i\v{c}ek, {\em Nonlinear Partial Differential Equations with Applications}, second ed., Internat. Ser. Numer. Math., {\bf 153}, Birkhäuser, Basel, Boston, Berlin (2013).
	
	\bibitem{Sc31} 
	J.~Schauder, {\em Potentialtheoretische Untersuchungen}, Math. Z. {\bf 33} (1931), 602-640.
	
	\bibitem{Sc32}
	J.~Schauder, {\em Bemerkung zu meiner Arbeit ``Potentialtheoretische Untersuchungen I (Anhang)''}, Math. Z. {\bf 35} (1932), 536--538. 
	
	\bibitem{SoZo92}
	J. Sokolowski, J.P. Zol\'esio, {\em Introduction to Shape Optimization. Shape Sensitivity Analysis}, 
	Springer-Verlag, Berlin (1992).
	
	\bibitem{Va88}
	T. Valent, {\em Boundary value problems of finite elasticity: local theorems on existence, uniqueness, and analytic dependence on data}, Springer-Verlag, New York (1988).
	
\end{thebibliography}
\end{document}